\documentclass[11pt,twoside]{article}
\usepackage[headings]{fullpage}
\usepackage{amsmath,amssymb}
\usepackage{fancyhdr}
\usepackage{times}

\usepackage{graphicx}
\usepackage[
            colorlinks=true,%
            breaklinks=true,%
            linkcolor=blue,
            urlcolor=blue,%
            citecolor=blue,%
            pdfauthor={S. S. Collis and M. Heinkenschloss},%
            bookmarksopen=false,%
            pdfpagemode=None]{hyperref}

\newtheorem{theorem}{Theorem}[section]
\newtheorem{lemma}[theorem]{Lemma}

\newenvironment{proof}{\noindent {\bf Proof:} }{\hfill $\square$ \\[2ex] }
\newenvironment{keywords}{\begin{quote} {\bf Key words} }
                         {\end{quote} }
\newenvironment{AMS}{\begin{quote} {\bf AMS subject classifications} }
                         {\end{quote} }

\newcommand{\set}[2]{\left\{{#1}\,:~{#2}\right\}}

\newcommand{\CL}{\ensuremath{\mathcal{L}} } %
\newcommand{\CT}{\ensuremath{\mathcal{T}} } %
\newcommand{\CX}{\ensuremath{\mathcal{X}} } %
\newcommand{\BK}{\ensuremath{\mathbf{K}} } %
\newcommand{\BR}{\ensuremath{\mathbf{R}} } %
\newcommand{\bc}{\ensuremath{\mathbf{c}}} %
\newcommand{\bn}{\ensuremath{\mathbf{n}}} %
\newcommand{\br}{\ensuremath{\mathbf{r}}} %
\newcommand{\bx}{\ensuremath{\mathbf{x}}} %
\newcommand{\bz}{\ensuremath{\mathbf{z}}} %
\def\frechet{Fr\'echet}

\newcommand {\real} {I\!\!R}
\newcommand {\nat} {I\!\!N}

\def\Pe{{\sf Pe}}

\newcommand{\deq}{\raisebox{0pt}[1ex][0pt]{$\stackrel{\scriptscriptstyle{\rm def}}{{}={}}$}}

\newcommand{\eref}[1]{\mbox{\rm(\ref{#1})}}
\newcommand{\tref}[1]{\mbox{\rm\ref{#1}}}

\newcommand {\half} {\mbox{$\frac{1}{2}$}}

\title{\bf Analysis of the Streamline Upwind/Petrov Galerkin Method
           Applied to the Solution of Optimal Control Problems
           \thanks{This work was supported in part by TX-ATP grant 
                   003604-0001-1999 and NSF grants DMS-0075731
                   and ACI-0121360. \newline
                   This report from 2002 has been posted on arXiv for easier access in 2024.}
       }
\author{ \and
         {\bf S. Scott Collis}
         \thanks{Department of Mechanical Engineering and Materials Science,
                 MS-321, Rice University, 6100 Main Street,
                 Houston, TX 77005-1892. E-mail: collis@rice.edu} 
         \and
         {\bf Matthias Heinkenschloss}
         \thanks{Department of Computational and Applied Mathematics,
                 MS-134, Rice University, 6100 Main Street,
                 Houston, TX 77005-1892.  E-mail: heinken@rice.edu} 
}
\date{March 2002}

\begin{document}

\numberwithin{equation}{section}
\numberwithin{figure}{section}
\numberwithin{table}{section}

\maketitle

\thispagestyle{empty}

\begin{abstract}
    We study the effect of the streamline upwind/Petrov Galerkin (SUPG) 
    stabilized finite element method on the discretization of optimal 
    control problems governed by linear advection-diffusion equations.
    We compare two approaches for the numerical solution of such
    optimal control problems.
    In the discretize-then-optimize approach the optimal control problem 
    is first discretized,
    using the SUPG method for the discretization of the advection-diffusion 
    equation, and then the resulting finite dimensional optimization 
    problem is solved.
    In the optimize-then-discretize approach one first computes the infinite 
    dimensional
    optimality system, involving the advection-diffusion equation as well
    as the adjoint advection-diffusion equation, and then discretizes
    this optimality system using the SUPG method for both the
    original and the adjoint equations. 
    These approaches lead to different results. The main result of this paper
    are estimates for the error between
    the solution of the infinite dimensional optimal control problem
    and their approximations computed using the previous approaches.
    For a class of problems prove that the optimize-then-discretize approach has 
    better asymptotic convergence properties if finite elements of order
    greater than one are used. For linear finite elements our theoretical
    convergence results for both approaches are comparable, except in the
    zero diffusion limit where again the optimize-then-discretize approach seems
    favorable.
    Numerical examples are presented to illustrate some of the theoretical results.
\end{abstract}



\section{Introduction}   \label{sec:intro}

This paper is concerned with the accuracy of numerical solutions 
of optimal control problems governed by the advection-diffusion equation.
Specifically, we are interested in the effect of the
streamline upwind/Petrov Galerkin (SUPG) stabilized finite
element method on the discretization of the optimal control problem.
To be more precise, 
we consider the linear quadratic optimal control problem
\begin{equation}  \label{eq:intro-obj}
    \min \frac{1}{2} \int_\Omega  (y(x) - \widehat{y}(x))^2 dx
    + \frac{\omega}{2}  \int_\Omega  u^2(x) dx
\end{equation}
subject to
\begin{subequations}     \label{eq:intro-state}
\begin{align} 
    - \epsilon \Delta y(x) + \bc(x) \cdot \nabla y(x) + r(x) y(x)
    &= f(x) + u(x) , & x &\in \Omega,\;  \\
    y(x) &= d(x),    & x &\in \Gamma_d, \\
   \epsilon \frac{\partial}{\partial \bn}  y(x) &= g(x), 
                     & x &\in \Gamma_n,
\end{align}
\end{subequations}
where $\Gamma_d \cap \Gamma_n = \emptyset$, 
$\Gamma_d \cup \Gamma_n = \partial \Omega$,
$\bc, d, f, g, r, \widehat{y}$ are given functions,
$\epsilon, \omega > 0$ are given scalars,
and $\bn$ denotes the outward unit normal. 
Assumptions on these data that ensure the well-posedness of 
the problem will be given in the next section. 

For advection dominated problems the standard  Galerkin finite element
method applied to the state equation \eref{eq:intro-state} produces
strongly oscillatory solutions, unless the mesh size $h$ is chosen
sufficiently small relative to $\epsilon / \| \bc(x) \|$, $x \in \Omega$.
To produce better approximations to the solution of \eref{eq:intro-state}
for modest mesh sizes, various augmentations of the standard 
Galerkin finite element method have been proposed.
For an overview see 
\cite{KWMorton_1996,AQuarteroni_AValli_1994,HGRoos_MStynes_LTobiska_1996}.
In this paper we focus on the streamline upwind/Petrov Galerkin (SUPG) method 
of Hughes and Brooks \cite{ANBrooks_TJRHughes_1982a}. 
The SUPG method adds to the weak form
of the state equation \eref{eq:intro-state} a term with the properties that 
(a) the weak form of the modification has better stability properties  
than the bilinear form associated with \eref{eq:intro-state} and
(b) the added term evaluated at the exact solution of \eref{eq:intro-state} 
vanishes. Because of these properties the SUPG method is called a strongly 
consistent stabilization method \cite{AQuarteroni_AValli_1994}.

For the numerical solution of the optimal control problem there are at 
least two approaches. In the first approach,  called the 
{\em optimize-then-discretize} approach, one first derives
the optimality conditions for  \eref{eq:intro-obj}, 
\eref{eq:intro-state}. In Section \eref{sec:model-exist} we will see that
the optimality conditions consist of the state equation \eref{eq:intro-state},
the adjoint partial differential equation (PDE)
\begin{subequations}   \label{eq:intro-adj}
\begin{align}
\label{eq:intro-adj-a}
    - \epsilon \Delta \lambda(x) - \bc(x)  \cdot \nabla \lambda(x)
       + (r(x) - \nabla \cdot \bc(x) ) \lambda(x)
    &= -(y(x) - \hat{y}(x)), & x &\in \Omega, \\
\label{eq:intro-adj-b}
     \lambda(x) &= 0, & x &\in \Gamma_d, \\
\label{eq:intro-adj-c}
  \epsilon  \frac{\partial}{\partial \bn}  \lambda(x)
  + \bc(x) \cdot \bn(x) \; \lambda(x)
                &= 0,  & x &\in \Gamma_n
\end{align}
\end{subequations}
and the gradient equation 
\begin{eqnarray}  \label{eq:intro-grad}
   \lambda(x)  =  \omega u(x) \qquad x \in \Omega.
\end{eqnarray}
Then one discretizes each equation
\eref{eq:intro-state}, \eref{eq:intro-adj} and \eref{eq:intro-grad}, 
using possibly different discretization schemes for each one. 
Since the adjoint equation \eref{eq:intro-adj} is also an 
advection-diffusion equation, but with advection $- \bc$,
we discretize it using the SUPG method. If we proceed this
way, the optimize-then-discretize approach leads to a discretization
of the optimality system \eref{eq:intro-state}, \eref{eq:intro-adj},
\eref{eq:intro-grad} that is strongly consistent. However,
this discretization of the optimality system \eref{eq:intro-state}, 
\eref{eq:intro-adj}, \eref{eq:intro-grad} leads to a nonsymmetric
linear system, which implies that there is no finite dimensional
optimization problem for which
this discretization of 
\eref{eq:intro-state}, \eref{eq:intro-adj}, \eref{eq:intro-grad} 
is the optimality system. The details of the optimize-then-discretize 
approach will be discussed in Section~\tref{sec:opt-then-disc}.
In the other approach for the numerical solution of \eref{eq:intro-obj}, 
\eref{eq:intro-state} called the {\em discretize-then-optimize} 
approach, one first discretizes the state equation using SUPG
and the objective function and then solves the resulting finite dimensional
optimization problem. The optimality conditions of this finite dimensional
optimization problem contain equations, which we call the
discrete adjoint equation and the discrete gradient equation,
that  can be viewed as discretizations of \eref{eq:intro-adj} and 
\eref{eq:intro-grad}, respectively.
The SUPG stabilization term added to the
state equation \eref{eq:intro-state} produces a contribution to
the discrete adjoint equation and to the discrete gradient equation. 
This contribution to the discrete adjoint equation has a stabilizing
effect, but the discrete adjoint equation is in general {\em not} a 
strongly consistent stabilization method for \eref{eq:intro-adj}.
We will give a detailed discussion of the optimize-then-discretize 
approach in Section~\tref{sec:disc-then-opt}.
The main goal of this paper is to derive estimates of the error between
the solution $y, u, \lambda$ of the infinite dimensional optimality system 
\eref{eq:intro-state}, \eref{eq:intro-adj}, \eref{eq:intro-grad}
and their approximations computed using both, the 
discretize-then-optimize as well as the optimize-then-discretize
approach. Such error estimates will be provided in 
Section~\tref{sec:error}. Section~\tref{sec:numer} contains
a few numerical results that illustrate our theoretical findings.

We will see that even in our simple model problem \eref{eq:intro-obj}, 
\eref{eq:intro-state} differences can arise between the 
discretize-then-optimize and the optimize-then-discretize approach.
It is important to understand and analyze these to better assess
the implication of numerical solution approaches to much more
complicated optimal control or optimal design problems
that involve nonlinear state equations solved using stabilization techniques.
We also note that the general issues described here for the SUPG stabilization
also arise when other stabilizations are used, such as the 
Galerkin/Least-squares (GLS) method of
Hughes, Franca and Hulbert \cite{TJRHughes_LPFranca_GMHulbert_1989a}
and the stabilization method of
Franca, Frey and Hughes \cite{LPFranca_SLFrey_TJRHughes_1992a}.

Throughout this paper we use the following notation for norms and inner 
products.
We define $\langle f, g \rangle_G = \int_G f(x) g(x) dx$,
$\| v \|_{0,\infty,G}   = \mbox{ ess sup}_{x \in G}  |v(x)|$
or $\| {\bf v} \|_{0,\infty,G}   
= \mbox{ ess sup}_{x \in G}  \sqrt{ \sum_i v_i(x)^2 }$
for vector valued ${\bf v}$, and
\[
   \| v \|_{k,G} =  \left( \sum_{| \alpha | \le k} 
                           \int_G ( \partial^\alpha v(x) )^2 dx \right)^{1/2}, 
  \qquad
    | v  |_{k,G} =  \left( \sum_{| \alpha | = k} 
                           \int_G ( \partial^\alpha v(x) )^2 dx \right)^{1/2}, 
\]
where $G \subset \Omega \subset \real^d$ or $G \subset \partial \Omega$
and $\alpha \in \nat_0^d$ is a multi-index, 
$| \alpha | = \sum_{i=1}^d \alpha_i$, and 
$\partial^\alpha = \partial^{\alpha_1} \ldots \partial^{\alpha_d}$.
If $G = \Omega$ we omit $G$ and simply write $\langle f, g \rangle$, etc.

\section{A Model Problem}   \label{sec:model}

%%%%%%%%%%%%%%%%%%%%%%%%%%%%%%%%%%%%%%%%%%%%%%%%%%%%%%%%%%%%%%%%%%%%%%%%%%%%
\subsection{Existence, Uniqueness and Characterization of Optimal Controls}
\label{sec:model-exist}

We define the state and control space
\begin{equation}   \label{eq:YU-def}
  Y = \set{ y \in H^1(\Omega)}{ y = d \mbox{ on } \Gamma_d}, \quad
  U = L^2(\Omega)
\end{equation}
and space of test functions
\begin{equation}   \label{eq:V-def}
     V = \set{ v \in H^1(\Omega)}{ v = 0 \mbox{ on } \Gamma_d}.
\end{equation}
The weak form of the state equations \eref{eq:intro-state}
is given by
\begin{equation}  \label{eq:model-state-weak-0}
   a(y,v) + b(u,v) = \langle f, v \rangle + \langle g, v \rangle_{\Gamma_n}
  \quad \forall v \in V,
\end{equation}
where
\begin{eqnarray}
\label{eq:model-a}
   a(y,v) &=& \int_\Omega  \epsilon  \nabla y(x) \cdot \nabla v(x)
                 + \bc(x) \cdot \nabla y(x) v(x)
                 + r(x) y(x) v(x) dx, \\
\label{eq:model-b}
   b(u,v) &=& - \int_\Omega  u(x) v(x) dx, \\
\label{eq:model-f}
   \langle f, v \rangle &=& \int_\Omega f(x) v(x) dx, \qquad
  \langle g, v \rangle_{\Gamma_n} = \int_{\Gamma_n} g(x) v(x) dx. 
\end{eqnarray}

We are interested in the solution of the optimal control problem
\begin{subequations}     \label{eq:model-ocp}
\begin{eqnarray}
\label{eq:model-obj}
    \mbox{minimize} &&
          \frac{1}{2} \|  y - \widehat{y} \|_0^2
           + \frac{\omega}{2} \| u \|_0^2 ,              \\[1ex]
\label{eq:model-state}
    \mbox{subject to} &&
       a(y,v) + b(u,v) = \langle f, v \rangle + \langle g, v \rangle_{\Gamma_n}
        \quad \forall v \in V, \\
       && y \in Y, u \in U.  \nonumber
\end{eqnarray}
\end{subequations}

We assume that 
\begin{subequations}     \label{eq:model-cr}
\begin{equation}  \label{eq:model-cr-a}
    f, \widehat{y} \in L^2(\Omega), 
    \bc \in \left( W^{1,\infty}(\Omega) \right)^2,
    r \in L^{\infty}(\Omega),
    d \in H^{3/2}(\Gamma_d),
    g \in H^{1/2}(\Gamma_n),
    \omega > 0, 
    \epsilon > 0,
\end{equation}
\begin{equation}  \label{eq:model-cr-a1}
   \Gamma_n \subset  \set{x \in \partial \Omega}{ \bc(x) \cdot \bn(x) \ge 0 }
\end{equation}
and
\begin{equation}  \label{eq:model-cr-b}
    r(x) - \half \nabla \cdot \bc(x) \ge r_0 > 0 \mbox{ a.e. in } \Omega.
\end{equation}
If $\Gamma_d \not= \emptyset$, there exists $\alpha > 0$ such that
$| y |_1 \le \alpha \| y \|_1$ for all $y \in V$ and
\eref{eq:model-cr-b} can be replaced by
\begin{equation}  \label{eq:model-cr-c}
    r(x) - \half \nabla \cdot \bc(x) \ge r_0 \ge  0 \mbox{ a.e. in } \Omega.
\end{equation}
\end{subequations}
For the well-posedness of the optimal control problem it is
sufficient to impose fewer regularity
requirements on the coefficient functions than those stated in
\eref{eq:model-cr-a}. We assume \eref{eq:model-cr-a} 
to establish convergence estimates for the SUPG finite element method.

Under the assumptions \eref{eq:model-cr}, the bilinear form
$a$ is continuous on $V \times V$ and $V$-elliptic. In fact,
$a(y,y) \ge \epsilon \| \nabla y \|_0^2 + r_0  \| y \|_0^2$ for all $y \in V$
(e.g., \cite[p.~165]{AQuarteroni_AValli_1994} or 
\cite[Sec.~2.5]{KWMorton_1996}).
Hence the theory in \cite[Sec.~II.1]{JLLions_1971} guarantees
the existence of a unique solution $(y,u) \in Y \times U$ of
\eref{eq:model-ocp}.

\begin{theorem}     \label{th:ocp-sol}
  If \eref{eq:model-cr} are satisfied,
  the optimal control problem \eref{eq:model-ocp}
  has a unique solution $(y, u) \in Y \times U$.
\end{theorem}

The theory in \cite[Sec.~II.1]{JLLions_1971} also provides necessary and
sufficient optimality conditions, which can be best described using the
Lagrangian
\begin{equation}  \label{eq:model-lag}
 L(y,u,\lambda)
 =  \frac{1}{2} \|  y - \widehat{y} \|_0^2
    + \frac{\omega}{2} \| u \|_0^2
    + a( y, \lambda) + b( u, \lambda)
      - \langle f, \lambda \rangle - \langle g, v \rangle_{\Gamma_n}.
\end{equation}
The necessary and, for our model problem,
sufficient optimality conditions can be obtained by setting
the partial \frechet -derivatives of \eref{eq:model-lag} with
respect to states $y$, controls $u$ and adjoints $\lambda$ equal
to zero. This gives the following system consisting of\\
the adjoint equation
\begin{subequations}  \label{eq:model-opt-cond}
\begin{equation}  \label{eq:model-adj-weak}
  a( \psi, \lambda) = - \langle  y - \hat{y}, \psi \rangle
     \quad \forall \psi \in V,
\end{equation}
the gradient equation
\begin{eqnarray}  \label{eq:model-grad-weak}
   b( w, \lambda) + \omega \langle u,  w \rangle = 0
     \quad \forall w \in U,
\end{eqnarray}
and the state equation
\begin{equation}  \label{eq:model-state-weak}
   a( y, v) + b( u, v) = \langle f, v \rangle + \langle g, v \rangle_{\Gamma_n}
   \quad \forall v \in V.
\end{equation}
\end{subequations}
The gradient equation \eref{eq:model-grad-weak}
simply means that $\lambda(x)  =  \omega u(x)$, $x \in \Omega$
(cf.\ \eref{eq:intro-grad})
and \eref{eq:model-adj-weak} is the weak form of \eref{eq:intro-adj}.

The adjoint equation \eref{eq:intro-adj} is also an advection-diffusion
equation, but advection is now given by $- \bc$ and the reaction term is
$r - \nabla \cdot \bc$.

The convergence theory for SUPG methods requires that the solution
$y, u, \lambda$ is more regular than indicated by Theorem~\tref{th:ocp-sol}.
This can be guaranteed if the problem data are such that the
state equation  \eref{eq:intro-state} and adjoint equation 
\eref{eq:intro-adj} admit more regular solutions. This motivates our
regularity assumptions \eref{eq:model-cr-a} on the data.  
The following result is an application of \cite[Thm.~2.4.2.5]{PGrisvard_1985}
to \eref{eq:intro-state} and \eref{eq:intro-adj}.

\begin{theorem}     \label{th:ocp-sol-reg}
  Let $\Omega$ be a bounded open subset of $\real^n$ with a
  $C^{1,1}$ boundary and $\Gamma_d = \partial \Omega$.
  If the assumption \eref{eq:model-cr-a} is satisfied and $r \ge r_0>0$ a.e.,
  then the unique solution of the optimal control problem \eref{eq:model-ocp}
  and the associated adjoint satisfy
  $y \in H^2(\Omega), u \in H^2(\Omega), \lambda \in H^2(\Omega)$.
\end{theorem}

%%%%%%%%%%%%%%%%%%%%%%%%%%%%%%%%%%%%%%%%%%%%%%%%%%%%%%%%%%%%%%%%%%%%%%

\subsection{Discretization of the State Equations}   \label{sec:supg-state}

For the discretization of the state equation we use conforming finite
elements. We let $\{ \CT_h \}_{h > 0}$ be a family of
quasi-uniform triangulations of $\Omega$  \cite{PGCiarlet_1991}.
To approximate the state equation we use the spaces
\begin{equation}   \label{eq:Vh-def}
\begin{array}{rcl}
  Y_h &=& \set{ y_h \in Y}{ y_h|_T \in P_k(T) \mbox{ for all } T \in \CT_h}, \\
  V_h &=& \set{ v_h \in V}{ v_h|_T \in P_k(T) \mbox{ for all } T \in \CT_h },
  \quad k \ge 1.
\end{array}
\end{equation}

For advection dominated problems the standard Galerkin 
method applied
to the state equation \eref{eq:model-state-weak-0} produces 
strongly oscillatory approximations, unless the mesh size $h$ is chosen 
sufficiently small relative to $\epsilon/ \| \bc \|_{0,\infty}$.
To obtain approximate solutions of better quality on coarser meshes,
various stabilization techniques have been proposed. 
For an overview see \cite[Secs.~8.3.2,8.4]{AQuarteroni_AValli_1994}
or \cite[Sec.3.2]{HGRoos_MStynes_LTobiska_1996}.
We are interested in the streamline upwind/Petrov Galerkin (SUPG) method of 
Hughes and Brooks \cite{ANBrooks_TJRHughes_1982a}.
The SUPG method computes an approximation
$y_h \in Y_h$ of the solution $y$ of the state equation
\eref{eq:model-state} by solving
\begin{equation}  \label{eq:model-disc-state00}
   a_h^{\sf s}(y_h,v_h) + b_h^{\sf s}(u_h,v_h)
    = \langle f, v_h \rangle_{h}^{\sf s} + \langle g, v_h \rangle_{\Gamma_n}
  \quad \forall v_h \in V_h,
\end{equation}
where
\begin{subequations} \label{eq:model-a-b}
\begin{eqnarray}  
   a_h^{\sf s}(y,v_h) &=& a(y,v_h)
      + \sum_{T_e \in \CT_h} \tau_e
          \langle - \epsilon \Delta y + \bc \cdot \nabla y + r y,
                 \bc \cdot \nabla v_h \rangle_{T_e}, \\
   b_h^{\sf s}(u,v_h) &=&
    - \int_\Omega  u(x) v_h(x) dx
    - \sum_{T_e \in \CT_h} \tau_e  \langle u, \bc \cdot \nabla v_h \rangle_{T_e},\\
   \langle f, v_h \rangle_h^{\sf s} &=& \langle f, v_h \rangle
   + \sum_{T_e \in \CT_h} \tau_e  \langle f, \bc \cdot \nabla v_h\rangle_{T_e}.
\end{eqnarray}
\end{subequations}
In \eref{eq:model-state-disc0}
the superscript ${\sf s}$ is used to indicate that the stabilization method is
applied to the state equation, i.e., all parameters in the stabilization method
are based on information from the state equation only.
The reason for the additional superscript ${\sf s}$ will become apparent
in Section \tref{sec:opt-then-disc}.
The addition of the term 
$\sum_{T_e \in \CT_h} \tau_e
  \langle - \epsilon \Delta y + \bc \cdot \nabla y + r y,
         \bc \cdot \nabla v_h \rangle_{T_e}$ to $a(y,v_h)$ 
introduces additional element wise diffusion 
$\langle \bc \cdot \nabla y, \bc \cdot \nabla v_h \rangle_{T_e}$ 
and enhances the stability properties of $a(y,v_h)$
(see Lemma~\tref{l:ah-stab-est} below).
The terms in \eref{eq:model-disc-state00} added to the standard Galerkin 
formulation are such that the exact solution $y$ of \eref{eq:intro-state}
satisfies  \eref{eq:model-disc-state00}, provided $y \in H^2(T_e)$,
$T_e \in {\cal T}_h$. We will review the error estimates for the
SUPG method in Section~\tref{sec:supg-state-error}.

%%%%%%%%%%%%%%%%%%%%%%%%%%%%%%%%%%%%%%%%%%%%%%%%%%%%%%%%%%%%%%%%%%%%%%

\subsection{Discretization of the Optimization Problem}
\label{sec:disc-then-opt}

A frequently used approach for the numerical solution of an
optimal control problem is to discretize the optimal control
problem and to solve the resulting nonlinear
programming problem using a suitable optimization algorithm.
This is also called the {\em discretize-then-optimize} approach.
In this scenario, the discretization of the optimal control
problem typically follows discretization techniques used for the
governing state equations.

In our problem we select the spaces \eref{eq:Vh-def} for the discretization
of the state and
\begin{equation}   \label{eq:Uh-def}
    U_h = \set{ u_h \in U}{ u_h|_T \in P_m(T) \mbox{ for all } T \in \CT_h },
    \quad m \ge 0,
\end{equation}
for the control.
To discretize the state equation, we apply the SUPG method.
The discretized optimal control problem is given by
\begin{subequations}     \label{eq:model-ocp-disc}
\begin{eqnarray}
\label{eq:model-obj-disc}
    \mbox{minimize} &&
          \frac{1}{2} \|  y_h - \widehat{y} \|_0^2
           + \frac{\omega}{2} \| u_h \|_0^2 ,              \\[1ex]
\label{eq:model-state-disc0}
    \mbox{subject to} &&
    a_h^{\sf s}(y_h,v_h) + b_h^{\sf s}(u_h,v_h) 
    = \langle f, v_h \rangle_h^{\sf s} + \langle g, v_h \rangle_{\Gamma_n}
   \quad \forall v_h \in V_h,  \\
        &&   y_h \in Y_h, \quad u_h \in U_h,  \nonumber
\end{eqnarray}
\end{subequations}
where $a_h^{\sf s}(y,v_h)$, $b_h^{\sf s}(u,v_h)$,
$\langle f, v_h \rangle_h^{\sf s}$ are defined in
\eref{eq:model-a-b}.

The Lagrangian for the discretized problem \eref{eq:model-ocp-disc}
is given by
\begin{equation}  \label{eq:model-lag-disc}
 L_h(y_h,u_h,\lambda_h)
 =  \frac{1}{2} \|  y_h - \widehat{y} \|_0^2
    + \frac{\omega}{2} \| u_h \|_0^2
    + a_h^{\sf s}( y_h, \lambda_h) + b_h^{\sf s}( u_h, \lambda_h)
      - \langle f, \lambda_h \rangle_h^{\sf s}
      - \langle g, v \rangle_{\Gamma_n},
\end{equation}
where $y_h \in Y_h, u_h \in U_h$ and $\lambda_h \in \Lambda_h \deq V_h$.
The necessary and sufficient optimality conditions for
the discretized problem are obtained by setting
the partial derivatives of \eref{eq:model-lag-disc}
to zero. This gives the following system consisting of\\
the discrete adjoint equations
\begin{subequations}  \label{eq:model-disc-opt-cond}
\begin{equation}  \label{eq:model-disc-adj}
   a_h^{\sf s}(\psi_h,\lambda_h)
  = - \langle y_h - \hat{y}, \psi_h \rangle
     \quad \forall \psi_h \in V_h,
\end{equation}
the discrete gradient equations
\begin{eqnarray}  \label{eq:model-disc-grad}
   b_h^{\sf s}(w_h,\lambda_h) + \omega \langle u_h, w_h \rangle = 0
     \quad \quad \forall w_h \in U_h,
\end{eqnarray}
and the discretized state equations
\begin{equation}  \label{eq:model-disc-state}
   a_h^{\sf s}(y_h,v_h) + b_h^{\sf s}(u_h,v_h)
   = \langle f, v_h \rangle^{\sf s}_{h} + \langle g, v \rangle_{\Gamma_n}
  \quad \forall v_h \in V_h.
\end{equation}
\end{subequations}
We use {\em discrete adjoint equations} and
{\em discrete gradient equations} to mean that these are the
adjoint and gradient equations for the discretized problem
\eref{eq:model-ocp-disc}.
We will use the phrases {\em discretized adjoint equations} and
{\em discretized gradient equations} to refer to discretizations
of the adjoint equation \eref{eq:intro-adj} and gradient equation
\eref{eq:intro-grad}, respectively.
As we will see in the next section, there are significant differences
between the discrete adjoint equations and the discretized adjoint equations
as well as between the discrete gradient equations and the
discretized gradient equations.

We notice that the discretized state equation \eref{eq:model-disc-state} 
is strongly consistent in the sense that \eref{eq:model-disc-state} 
is satisfied if $u_h, y_h$ are replaced by the optimal control $u$ and
the corresponding optimal state $y$.
However, strong consistency is lost in the discrete adjoint equations
\eref{eq:model-disc-adj} and the discrete gradient equations
\eref{eq:model-disc-grad}. 
Specifically, 
\begin{eqnarray*} 
   a_h^{\sf s}(\psi_h,\lambda_h) = a(\psi_h,\lambda_h)
      + \sum_{T_e \in \CT_h} \tau_e^{\sf s}
          \langle - \epsilon \Delta \psi_h + \bc \cdot \nabla \psi_h + r \psi_h,
                 \bc \cdot \nabla \lambda_h \rangle_{T_e}.
\end{eqnarray*}
The amount 
$\sum_{T_e \in \CT_h} \tau_e^{\sf s} 
\langle  -\bc \cdot \nabla \psi_h,
                 -\bc \cdot \nabla \lambda_h \rangle_{T_e}$
of streamline diffusion added to $a(\psi_h,\lambda_h)$ in the adjoint
equation appears to be right in the sense that this
amount (although possibly with a different $\tau_e^{\sf s} $) would be
added if the SUPG method had been applied to the adjoint equation
\eref{eq:intro-adj}. 
However, \eref{eq:model-disc-adj} is not satisfied if
$y_h, \lambda_h$ are replaced by the optimal state $y$ and corresponding
adjoint $\lambda$.
This lack of strong consistency is due to the fact that the discrete
adjoint equation is not a method of weighted residuals for the
continuous adjoint problem.  In particular, the resulting stabilization
term is not a weighted residual of the continuous adjoint equation on
element interiors since the diffusion, reaction, and source terms are
not accounted for.
Similarly, $b_h^{\sf s}$ in the gradient equation contains terms
that arise from the stabilization of the state equation and 
\eref{eq:model-disc-grad} is not satisfied if
$u_h, \lambda_h$ are replaced by the optimal control $u$ and corresponding
adjoint $\lambda$.

%%%%%%%%%%%%%%%%%%%%%%%%%%%%%%%%%%%%%%%%%%%%%%%%%%%%%%%%%%%%%%%%%%%%%%

\subsection{Discretization of the Optimality Conditions}
\label{sec:opt-then-disc}

Alternatively to the discretize-then-optimize approach
discussed in the previous section, one can obtain an
approximate solution of the optimal control problem by
tackling the optimality system \eref{eq:model-opt-cond}
directly. This leads to the {\em optimize-then-discretize} approach.
Here each equation in \eref{eq:model-opt-cond} is discretized
using a potentially different scheme.
In our case, we will use the same triangulation for
all three equations and we will use the state space \eref{eq:Vh-def}
and the control space \eref{eq:Uh-def} for the discretization of states
and controls, respectively, and we will use 
\begin{equation}   \label{eq:Lambdah-def}
  \Lambda_h = \set{ v_h \in V}{ v_h|_T \in P_\ell(T) \mbox{ for all } 
                                T \in \CT_h },
  \quad \ell \ge 1,
\end{equation}
for the discretization of the adjoints. 
It is possible to choose $\ell \not= k$.
Now we take into account that the adjoint equation
\eref{eq:intro-adj} is also an advection dominated
problem, but with advective term $- \bc \cdot \nabla \lambda$.
We discretize \eref{eq:intro-adj} using the SUPG method.
This leads to the discretized adjoint equations
\begin{subequations}  \label{eq:model-opt-cond-disc}
\begin{equation}  \label{eq:model-adj-disc}
  a_h^{\sf a}(\psi_h,\lambda_h)
  = -  \langle y_h - \hat{y}, \psi_h \rangle_h^{\sf a}
     \quad \forall \psi_h \in \Lambda_h,
\end{equation}
where
\begin{eqnarray}     \label{eq:model-a-adj}
   a_h^{\sf a}(\psi_h,\lambda)
   &=& a(\psi_h,\lambda)
      + \sum_{T_e \in \CT_h} \tau_e^{\sf a}
          \langle - \epsilon \Delta \lambda - \bc \cdot \nabla \lambda
                 + (r- \nabla \cdot \bc) \lambda,
                 -\bc \cdot \nabla \psi_h \rangle_{T_e}, \\
   \langle y - \hat{y}, \psi_h \rangle_h^{\sf a}
  &=& \langle y - \hat{y}, \psi_h \rangle
       + \sum_{T_e \in \CT_h} \tau_e^{\sf a}
                   \langle y - \hat{y}, -\bc \cdot \nabla \psi_h \rangle_{T_e}.
\end{eqnarray}
Here and in the
following the superscript ${\sf a}$ is used to indicate that the SUPG
method is applied to the adjoint equation, i.e., all parameters 
in the stabilization method applied to \eref{eq:intro-adj} are based on 
information from the adjoint equation \eref{eq:intro-adj}.

The gradient equation \eref{eq:model-grad-weak} is discretized using
\begin{eqnarray}  \label{eq:model-grad-disc}
   b( w_h, \lambda_h) + \omega \langle u_h, w_h \rangle = 0
     \quad \forall w_h \in U_h,
\end{eqnarray}
and the discretization of the state equations is identical to the
one used in the previous section, i.e.,
\begin{equation}  \label{eq:model-state-disc}
   a_h^{\sf s}(y_h,v_h) + b_h^{\sf s}(u_h,v_h)
   = \langle f, v_h \rangle^{\sf s}_{h} + \langle g, v_h \rangle_{\Gamma_n}
  \quad \forall v_h \in V_h.
\end{equation}
\end{subequations}
Unlike the discrete adjoint and gradient equations,
the discretized state, adjoint and gradient equations are strongly
consistent in the sense that if $y, u, \lambda$ solve
\eref{eq:model-opt-cond} and satisfy 
$y, \lambda \in H^2(T_e)$, for all $T_e \in \CT_h$, then
$y, u, \lambda$ also satisfy \eref{eq:model-opt-cond-disc}.

Due to the occurance of the SUPG terms in the right hand side
of \eref{eq:model-adj-disc} and in $b_h^{\sf s}$,
the discretization \eref{eq:model-opt-cond-disc} of the
infinite dimensional optimality conditions leads to a nonsymmetric
system for the computation of $y_h, u_h, \lambda_h$.
This implies that \eref{eq:model-opt-cond-disc} cannot be a system of
optimality conditions for an optimization problem, e.g., a 
perturbation of \eref{eq:model-ocp-disc}.

\section{Abstract Formulation}   \label{sec:absform}

To analyze the error between the solution of the
optimal control problem \eref{eq:model-ocp} and
the solution of the discretized optimal control problem
\eref{eq:model-ocp-disc} we could apply the approximation 
theory for saddle point problems described, e.g., in 
\cite{FBrezzi_MFortin_1991}. However, the optimize-then-discretize
approach leads to a non-symmetric system \eref{eq:model-opt-cond-disc}. 
Thus there is no optimization problem whose optimality system is
given by \eref{eq:model-opt-cond-disc} and the theory in
\cite{FBrezzi_MFortin_1991} can not be applied to this situation. 
We prefer to use a framework that is common in Numerical Analysis 
for the estimation of the approximation error in operator equations.
We give a brief review here and apply it in the following
section to our problem.

The necessary and sufficient optimality conditions 
\eref{eq:model-opt-cond} 
can be viewed as an operator equation 
\begin{equation}   \label{eq:Kx=r}
   \BK \bx = \br
\end{equation}
in $\CX^*$, where $\CX$ is a Banach space, $\CX^*$ is its dual
and $\BK \in \CL(\CX,\CX^*)$ is continuously invertible.
In the following section we describe in detail how \eref{eq:Kx=r}
relates to our problem.
The discretized problem is described by the equation
\begin{equation}   \label{eq:Khxh=rh}
   \BK_h \bx_h = \br_h,
\end{equation}
in $\CX_h^*$,where $\CX_h$ is a finite dimensional Banach space
with norm $\| \cdot \|_h$ and
$\BK_h \in \CL(\CX_h,\CX^*_h)$ is continuously invertible.

To derive an error estimate we 
let $\BR_h: \CX \rightarrow \CX_h$ be a restriction operator
and we consider the identity
\[
   \BK_h (\bx_h - \BR_h(\bx)) = \br_h - \BK_h \BR_h(\bx).
\]
We immediately obtain the estimate
\begin{equation}   \label{eq:abs-error-0} 
       \| \bx_h - \BR_h(\bx) \|_h 
   \le \| \BK_h^{-1} \|_h  \; \| \br_h - \BK_h \BR_h(\bx) \|_h,
\end{equation}
where $\| \BK_h^{-1} \|_h$ denotes the operator norm of $\BK_h^{-1}$
induced by $\| \cdot \|_h$.
If there exists $\kappa > 0$ independent of $h$ such that the
stability estimate
\begin{equation}   \label{eq:abs-stab} 
   \| \BK_h^{-1} \|_h  \le \kappa  \quad \mbox{ for all } h
\end{equation}
is valid and if we can prove a consistency result
of the form
\begin{equation}   \label{eq:abs-consist} 
   \| \br_h - \BK_h \BR_h(\bx) \|_h  = O(h^p),
\end{equation}
then we obtain $\| \bx_h - \BR_h(\bx) \|_h = O(h^p)$.
If $\| \cdot \|_h$ can be extended to define a norm
on $\CX$, and if we can prove
\begin{equation}   \label{eq:abs-interp} 
   \| \bx - \BR_h(\bx) \|_h  = O(h^q),
\end{equation}
then a simple application of the triangle inequality shows
\begin{equation}   \label{eq:abs-error} 
       \| \bx_h - \bx \|_h 
   \le \| \BK_h^{-1} \|_h  \| \br_h - \BK_h \BR_h(\bx) \|_h
       + \| \BR_h(\bx) - \bx \|_h
    = O(h^{\min\{p,q\}})
\end{equation}
In \eref{eq:abs-error} the error is measured in a norm that
depends on $h$. This is certainly not problematic if 
there exists $\eta > 0 $ independent of $h$ such that
$\| \bx_h - \bx \| \le \eta \| \bx_h - \bx \|_h$ for all $h$,
which will be true in our situation.

In our applications, $\CX_h = Y_h \times U_h \times \Lambda_h$ 
with some finite dimensional Banach spaces
$Y_h, U_h, \Lambda_h$ equipped with norms
$\| \cdot \|_{Y_h}$, $\| \cdot \|_{U_h}$ and
$\| \cdot \|_{\Lambda_h}$, respectively.
The norm $\| \cdot \|_h$ on $\CX_h$ is defined as
$\| \bx_h \|_h 
 = \| y_h \|_{Y_h} + \| u_h \|_{U_h} + \| \lambda_h \|_{\Lambda_h}$,
where $\bx = (y_h, u_h, \lambda_h)$.
Furthermore, in our applications the operator $\BK_h$ is of the form
\begin{equation}   \label{eq:Kh-def}
    \BK_h
   = \left( \begin{array}{ccc}
               H^{yy}_h  & H^{yu}_h & \widetilde{A}_h^* \\
               H^{uy}_h  & H^{uu}_h & \widetilde{B}_h^* \\
               A_h       & B_h      & 0
     \end{array}  \right) .
\end{equation}
The operator $\BK_h$ is not necessarily selfadjoint, i.e.,
we do {\em not} assume that $A_h^* = \widetilde{A}_h^*$, 
$B_h^* = \widetilde{B}_h^*$, $(H^{uy}_h)^* = H^{yu}_h$,
$(H^{yy}_h)^* = H^{yy}_h$, or $(H^{uu}_h)^* = H^{uu}_h$.
We assume, however, that $A_h$ and $\widetilde{A}_h^*$
are invertible.

To estimate $\| \BK_h^{-1} \|_h$ we consider
\begin{eqnarray}   \label{eq:K-fac}
   \left( \begin{array}{ccc}
               H^{yy}_h  & \widetilde{A}_h^* & H^{yu}_h \\
               A_h       & 0   & B_h      \\ 
               H^{uy}_h  & \widetilde{B}_h^* & H^{uu}_h
   \end{array}  \right)
 &=&  \left( \begin{array}{ccc}
               I  &  0 & 0 \\
               0  &  I & 0 \\  
          \widetilde{B}_h^* (\widetilde{A}_h^*)^{-1} & 
          ( H^{uy}_h - \widetilde{B}_h^* (\widetilde{A}_h^*)^{-1} H^{yy}_h)
               A_h^{-1}  & I
      \end{array}  \right)
      \left( \begin{array}{ccc}
               H^{yy}_h & \widetilde{A}_h^* & 0 \\
               A_h      & 0   & 0 \\  
               0        & 0   & \widehat{H}_h
      \end{array}  \right)    \nonumber \\
 && \times
     \left( \begin{array}{lcc}
            I \quad &  0 & A_h^{-1} B_h \\
            0  &  I & (\widetilde{A}_h^*)^{-1} ( H^{yu}_h - H^{yy}_h A_h^{-1} B_h)\\  
            0  &  0 & I
      \end{array}  \right),
\end{eqnarray}
where 
\begin{equation}   \label{eq:red-H-def}
   \widehat{H}_h 
   =  H^{uu}_h 
      - \widetilde{B}_h^* (\widetilde{A}_h^*)^{-1} H^{yu}_h
      - H^{uy}_h A_h^{-1} B_h
      + \widetilde{B}_h^* (\widetilde{A}_h^*)^{-1} H^{yy}_h A_h^{-1} B_h
\end{equation}   
The operator on the left hand side in \eref{eq:K-fac} is just a
symmetric permutation of $\BK_h$, which does not effect the invertibility
of $\BK_h$ or the estimate for $\| \BK_h^{-1} \|_h$.
Under the assumption that $A_h$ and $\widetilde{A}_h^*$ are invertible,
\eref{eq:K-fac} shows that $\BK_h$ is invertible
if and only if $\widehat{H}_h$is  invertible.
Using
\begin{equation}   \label{eq:K11-inv}
    \left( \begin{array}{cc}
        H^{yy}_h  & \widetilde{A}_h^*  \\
        A_h       & 0   
    \end{array}  \right)^{-1}
  = \left( \begin{array}{cc}
      0                    & A_h^{-1}  \\
      (\widetilde{A}_h^*)^{-1} & -(\widetilde{A}_h^*)^{-1} H^{yy}_h A_h^{-1}
    \end{array}  \right)
\end{equation}
and \eref{eq:K-fac} we immediately obtain the following result.

\begin{lemma}   \label{eq:K-stab}
\sloppy
   If $\widetilde{A}_h$, $\widetilde{A}_h^*$ and $\widehat{H}_h$ are 
   invertible for all $h$ and if 
   $\| A_h^{-1} \|_{{\cal L}(Y_h^*,Y_h)}$, 
   $\|(\widetilde{A}_h^*)^{-1}\|_{{\cal L}(\Lambda_h^*,\Lambda_h)}$,
   $\|\widehat{H}_h^{-1}\|_{{\cal L}(U_h,U_h^*)}$,
   $\| A_h^{-1} B_h\|_{{\cal L}(U_h,Y_h)}$, 
   $\|\widetilde{B}_h^* (\widetilde{A}_h^*)^{-1}\|_{{\cal L}(Y_h^*,U_h^*)}$,
   $\|( H^{uy}_h - \widetilde{B}_h^* 
      (\widetilde{A}_h^*)^{-1} H^{yy}_h) A_h^{-1}\|_{{\cal L}(Y_h^*,U_h^*)}$,
   $\|(\widetilde{A}_h^*)^{-1} ( H^{yu}_h - H^{yy}_h A_h^{-1} 
               B_h)\|_{{\cal L}(U_h,\Lambda_h)}$, and
   $\|(\widetilde{A}_h^*)^{-1} H^{yy}_h A_h^{-1} \|_{{\cal L}(Y_h,\Lambda_h)}$ 
   are uniformly bounded,
   then there exists  $\kappa > 0$ such that
   \eref{eq:abs-stab} holds.
\end{lemma}

\section{Error Estimates for the SUPG Method}   \label{sec:error}

In this section we derive estimates for the error between 
the solution of the optimal control problem
and the computed approximations using both,
the discretize-then-optimize and the optimize-then-discretize
approaches.

Before we apply the theory outlined in Section \tref{sec:absform}
to the optimization problem, we briefly review 
estimates for the error between the solution $y$ of
\eref{eq:model-state-weak-0}
and its approximation $y_h$ by the  SUPG method.
Such estimates are given in the paper 
\cite{CJohnson_UNavert_JPitkaranta_1984a} and the books
\cite{PKnabner_LAngermann_2000,HGRoos_MStynes_LTobiska_1996}.
See also \cite{AQuarteroni_AValli_1994}.
We sketch the main points of the error analysis to recall some basic 
estimates needed in our
analysis of the SUPG discretization for optimal control.

Throughout this section we assume that the Dirichlet boundary data 
are $d = 0$.
This can always be achieved by a shift of the state.

%%%%%%%%%%%%%%%%%%%%%%%%%%%%%%%%%%%%%%%%%%%%%%%%%%%%%%%%%%%%%%%%%%%%%%%%%%%
\subsection{Error Estimates for the State Equation}
\label{sec:supg-state-error}

We define
\begin{equation}   \label{eq:sd-norm}
    \| v \|_{SD}^2
  = \epsilon | v |_1^2
    + r_0 \| v \|_0^2
    + \sum_{T_e \in \CT_h} \tau_e \| \bc \cdot \nabla v \|_{0,T_e}^2.
\end{equation}

Recall that $k \ge 1$ is the polynomial degree of the finite element spaces 
$Y_h, V_h$ defined in \eref{eq:Vh-def}. 
For $y \in H^{k+1}(\Omega)$ we let $y^I$ be its $Y_h$-interpolant.
We recall the interpolation error estimate
\begin{equation}   \label{eq:int-error}
    | y - y^I |_{p,T_e} \le \mu_{\rm int} h_e^{k+1-p} |y|_{k+1,T_e}
    \quad \mbox{ for } p = 0, 1, 2
\end{equation}
and the inverse inequalities
\begin{equation}   \label{eq:inv-est}
    | v_h |_{1,T_e} \le \mu_{\rm inv} h_e^{-1} \|v_h\|_{0,T_e}, \quad
   \| \Delta v_h \|_{0,T_e} \le \mu_{\rm inv} h_e^{-1} \| \nabla v_h\|_{0,T_e},
   \quad \forall v_h \in V_h,
\end{equation}
see, e.g., \cite[Thms.~16.2, 17.2]{PGCiarlet_1991}.
Here $h_e$ denotes the radius of the cicumscribed circle of $T_e$ and
$h = \max_{T_e \in {\cal T}_h} h_e$.
The following lemma can be found, e.g., in
\cite[L.~3.28]{HGRoos_MStynes_LTobiska_1996} or
in \cite[pp.~325,326]{PKnabner_LAngermann_2000}.
\begin{lemma}   \label{l:ah-stab-est}
If
\begin{equation}   \label{eq:tau-ineq}
    0 < \tau_e^{\sf s}
     \le \min \left\{ \frac{ h_e^2 }{ \epsilon \mu_{\rm inv}^2 },
                     \frac{ r_0 }{ \| r \|_{0,\infty,T_e} }   \right\} ,
\end{equation}
then
\begin{equation}   \label{eq:ah-stable}
    a_h^{\sf s}(v_h,v_h) \ge  \half \| v_h \|_{SD}^2
     \quad \forall v_h \in V_h.
\end{equation}
\end{lemma}

The following inequalities can be found in
\cite[p.~232]{HGRoos_MStynes_LTobiska_1996} or
in \cite[pp.~327,328]{PKnabner_LAngermann_2000}.

\begin{lemma}     \label{eq:est}
  Let $y \in H^{k+1}(\Omega)$ with $k \ge 1$.
  There exists a constant $C > 0$ dependent on
  $\mu_{\rm int}, \bc, r$, but independent of $h_e, \tau_e$ such that
  \begin{eqnarray}
    \label{eq:est-1}
     \Big| \epsilon \langle \nabla (y^I - y) , \nabla v_h \rangle \Big|
     &\le& C \epsilon^{1/2} h^k | y |_{k+1} \; \| v_h \|_{SD} \\
    \label{eq:est-2}
     \Big| \langle \bc \cdot \nabla (y^I - y) + r (y^I - y) , v_h \rangle \Big|
     &\le& C h^k \left( \sum_{T_e \in \CT_h} ( 1+ 1/\tau_e^{\sf s} ) h_e^2
                                          | y |_{k+1,T_e} \right)^{1/2}
        \| v_h \|_{SD}
  \end{eqnarray}
  for all $v_h \in V_h$.
  Furthermore, if $\tau_e$ satisfies \eref{eq:tau-ineq}, then
  \begin{eqnarray}   \label{eq:est-3}
    \lefteqn{
    \Big| \sum_{T_e \in \CT_h}
      \tau_e \langle  - \epsilon \Delta (y^I - y)
                        + \bc \cdot \nabla (y^I - y)
                        + r (y^I - y) , \bc \cdot v_h \rangle  \Big| }
                                                             \nonumber\\
    &\le& C h^k \left( \sum_{T_e \in \CT_h} ( \epsilon + \tau_e^{\sf s} )
                                          | y |_{k+1,T_e}^2 \right)^{1/2}
        \| v_h \|_{SD}   \quad \mbox{ for all } v_h \in V_h.
   \end{eqnarray}
\end{lemma}

The stability result \eref{eq:ah-stable}, the
estimates \eref{eq:est-1}-\eref{eq:est-3}, and the identity
$a_h(y - y_h,v_h) = 0$ for all $v_h \in V_h$, yield
\begin{eqnarray}    \label{eq:supg-err-state00}
  \half \| y^I - y_h \|_{SD}^2
  &\le& a_h(y^I - y_h,y^I - y_h)  =  a_h(y^I - y,y^I - y_h)    \nonumber \\
  &\le& C h^k \left( \sum_{T_e \in \CT_h}
                    ( \epsilon + \tau_e^{\sf s}
                       + h_e^2/\tau_e^{\sf s} + h_e^2 )
                      | y |_{k+1,T_e}^2 \right)^{1/2}   \| y^I - y_h \|_{SD}.
\end{eqnarray}
The stabilization parameter $\tau_e$ is chosen to balance
the terms in $\epsilon + \tau_e^{\sf s} + h_e^2/\tau_e^{\sf s} + h_e^2$.
In particular, if
\begin{equation}      \label{eq:tau-def}
    \tau_e^{\sf s}
    = \left\{ \begin{array}{ll}
         \tau_1 \frac{h_e^2}{\epsilon}, & \Pe_e \le 1, \\
         \tau_2 h_e,                    & \Pe_e > 1,
       \end{array} \right.
\end{equation}
where $\tau_1, \tau_2 > 0$ are user specified constants and
\begin{equation}      \label{eq:peclet}
    \Pe_e = \frac{\| \bc \|_{0,\infty,T_e} h_e}{2 \epsilon}
\end{equation}
is the mesh P{\'e}clet number, then
\begin{equation}    \label{eq:supg-err-state0}
  \| y^I - y_h \|_{SD}
  \le   C h^k ( \epsilon^{1/2} +  h^{1/2}) | y |_{k+1} .
\end{equation}
An estimate of $\| y^I - y \|_{SD}$ using inequalities
similar to those in Lemma~\tref{eq:est} and an application of the triangle
inequality leads to the error estimate stated in the following theorem,
see \cite[Thm.~3.30]{HGRoos_MStynes_LTobiska_1996} or
\cite[Thm.~9.3]{PKnabner_LAngermann_2000}.

\begin{theorem}     \label{th:supg-err}
  Let \eref{eq:model-cr} be valid and
  let the solution $y$ of \eref{eq:model-state} satisfy
  $y \in H^{k+1}(\Omega)$ with $k \ge 1$.
  If $\tau_e$ satisfies \eref{eq:tau-ineq} and \eref{eq:tau-def},
  then the solution $y_h$ of \eref{eq:model-disc-state00} obeys
  \begin{equation}      \label{eq:supg-err-state}
       \| y - y_h \|_{SD}
       \le   C h^k ( \epsilon^{1/2} +  h^{1/2}) | y |_{k+1}.
  \end{equation}
\end{theorem}

%%%%%%%%%%%%%%%%%%%%%%%%%%%%%%%%%%%%%%%%%%%%%%%%%%%%%%%%%%%%%%%%%%%%%%%%%%%
\subsection{Error Estimates for the Optimal Control Problem}

We apply the framework of Section \tref{sec:absform} to our
problem. In this case,
\[
   \bx =  ( y,  u, \lambda ), \quad
   \CX = Y \times U  \times \Lambda,
\]
where $Y, U$ are defined in \eref{eq:YU-def} and
$\Lambda = V$ is specified in \eref{eq:V-def}.
Furthermore,
\[
   \CX_h = Y_h \times U_h  \times \Lambda_h,
\] 
where the discretized state and control spaces are given by
$Y_h, U_h$ are defined in \eref{eq:Vh-def} and \eref{eq:Uh-def},
respectively.
If we use the discretize-then-optimize approach,
the discrete adjoints are in $\Lambda_h = V_h$, where
$V_h$ is defined in \eref{eq:Vh-def}.
If we use the optimize-then-discretize approach,
then $\Lambda_h$ is defined in \eref{eq:Lambdah-def}.
The discrete state and  control spaces
will be equipped with norms
\[
   \| y_h \|_{Y_h}^2  
 = \| y_h \|_{SD}^2
 = \epsilon | y_h |_1^2
    + r_0 \| y_h \|_0^2
    + \sum_{T_e \in \CT_h} \tau_e^{\sf s} \| \bc \cdot \nabla y_h \|_{0,T_e}^2.
\]
and $\| \cdot \|_{U_h}  = \| \cdot \|_0$, respectively.
If the discretize-then-optimize approach is used, then
$\| \cdot \|_{\Lambda_h}  = \| \cdot \|_{Y_h} = \| \cdot \|_{SD}$.
If the optimize-then-discretize approach is used,
\[
   \| \lambda_h \|_{\Lambda_h}^2  
 = \| \lambda_h \|_{SD}^2
 = \epsilon | \lambda_h |_1^2
    + r_0 \| \lambda_h \|_0^2
    + \sum_{T_e \in \CT_h} \tau_e^{\sf a} 
             \| \bc \cdot \nabla \lambda_h \|_{0,T_e}^2.
\]
Since stabilization parameters $\tau_e^{\sf s}$
and $\tau_e^{\sf a}$ might be different in the discretization of
state and adjoint equation, it is not quite accurate to use
$\| \cdot \|_{SD}$ to denote both norms $\| \cdot \|_{Y_h}$ and
$\| \cdot \|_{\Lambda_h}$. However, we hope that its meaning is clear
from the context.
The space $\CX_h$ will be equipped with norm
\[
    \| \bx_h \|_h = \| y_h \|_{SD} + \| u_h \|_0 + \| \lambda_h \|_{SD},
\]
where $\bx_h = (y_h, u_h, \lambda_h)^T$.

The equation \eref{eq:Kx=r} corresponds to the
optimality conditions \eref{eq:model-opt-cond}.
Depending on whether the discretize-then-optimize approach
or the optimize-then-discretize approach is used, the
discrete equation \eref{eq:Khxh=rh} corresponds to 
\eref{eq:model-disc-opt-cond} or \eref{eq:model-opt-cond-disc},
respectively.

As the restriction operator $\BR_h: \CX \rightarrow \CX_h$, we choose
\[
   \BR_h(\bx)
      = \left(\begin{array}{c}
           y^I \\
           P u \\
           \lambda^I
         \end{array} \right),
\]
where $y^I, \lambda^I$ denote the interpolants of $y, \lambda$
onto $Y_h, \Lambda_h$ and where $P: U \rightarrow U_h$ is the
$L^2$-projection defined by
\begin{equation}    \label{eq:Proj-def}
    \langle Pu, w_h \rangle = \langle u, w_h \rangle
    \quad \forall w_h \in U_h.
\end{equation}
If $m \ge 1$ and if $u \in H^{m+1}(\Omega)$, then the optimality of
the projection $P$ and the interpolation estimate \eref{eq:int-error}
imply that
\begin{equation}  \label{eq:Pu-est}
    \| u - Pu \|_0 \le  \| u - u^I \|_0 \le \mu_{\rm int} h^{m+1} |u|_{m+1},
\end{equation}
where $u^I$ is the $U_h$-interpolant of $u$.

Recall that $k, \ell \ge 1$ and $m \ge 0$ are the polynomial
degrees of the finite element spaces 
$Y_h, \Lambda_h, U_h$ defined in \eref{eq:Vh-def}, \eref{eq:Lambdah-def}
and \eref{eq:Uh-def}, respectively. If the discretize-then-optimize
approach is used, $\Lambda_h = V_h$ and we set $\ell = k$.

The following lemma provides an estimate for the term 
$\| \BR_h(\bx) - \bx \|_h$ in the abstract error estimate
\eref{eq:abs-error}.

\begin{lemma}  \label{l:x-Rx}
  Let $\bx =  ( y, u, \lambda )$ be the solution
  of \eref{eq:model-ocp}. 
  If $k, \ell, m \ge 1$ and 
  $y \in H^{k+1}(\Omega)$, $u \in H^{m+1}(\Omega)$ and
  $\lambda \in H^{\ell+1}(\Omega)$,
  then there exists a constant $C$ depending on $\mu_{\rm int}, \bc, r$
  such that
  $\| y - y^I \|_{SD} \le C h^k ( \epsilon^{1/2} +  h^{1/2} ) |y|_{k+1}$,
  $\| \lambda - \lambda^I \|_{SD} 
        \le C h^\ell ( \epsilon^{1/2} +  h^{1/2} )  |\lambda|_{\ell+1}$ and
  $\| u - P u \|_0 \le C  h^{m+1} |u|_{m+1}$ for all $h$.
\end{lemma}
\begin{proof}
  The estimates for $\| y - y^I\|_{SD}, \| \lambda - \lambda^I\|_{SD}$
  follow from the interpolation estimate \eref{eq:int-error} using
  standard arguments, see \cite[Thm.~3.30]{HGRoos_MStynes_LTobiska_1996} or
  \cite[Thm.~9.3]{PKnabner_LAngermann_2000}. The estimate for 
  $\| u - Pu \|_0$ is shown in \eref{eq:Pu-est}.
\end{proof}

Note that $u \in H^{m+1}(\Omega)$, $\lambda \in H^{\ell+1}(\Omega)$
and the optimality condition \eref{eq:intro-grad} imply that
$\lambda = \omega u \in H^{\min\{\ell+1,m+1\}}(\Omega)$. However, in 
Lemma~\tref{l:x-Rx} and in the following we prefer
to impose the regularity assumption on $\lambda$ and $u$ seperately, to
better indicate where each is used.

%%%%%%%%%%%%%%%%%%%%%%%%%%%%%%%%%%%%%%%%%%%%%%%%%%%%%%%%%%%%%%%%%%%%%%%%%%%
\subsection{Discretize-Then-Optimize}

In the discretize-then-optimize approach the
discrete equation \eref{eq:Khxh=rh} corresponds to 
\eref{eq:model-disc-opt-cond}. The components of
$\BK_h$ in \eref{eq:Kh-def} are given by
\begin{equation}   \label{eq:Kh-def-disc-opt}
\begin{array}{c}
        \langle H^{yy}_h y_h, v_h \rangle_{V_h^* \times  V_h}
      = \langle y_h, v_h \rangle, \quad
        \langle H^{uu}_h u_h, w_h \rangle_{U_h^* \times U_h}
      = \omega \langle u_h, w_h \rangle, \quad
      H^{uy}_h = H^{yu}_h = 0  \\[1ex]
      \langle A_h y_h, v_h \rangle_{V_h^* \times V_h}
      = a_h^{\sf s}(y_h, v_h), \quad
      \langle B_h u_h, v_h \rangle_{V_h^*\times V_h}
      = b_h^{\sf s}(u_h, v_h), \quad
      \widetilde{A}_h = A_h, \quad
      \widetilde{B}_h = B_h.
\end{array}
\end{equation}
In particular, $\BK_h$ is selfadjoint.  
       
The next result establishes a stability estimate for the 
optimal control problem.

\begin{lemma}  \label{l:disc-opt-stab}
  Let $k \ge 1$ and suppose the solution 
  $\bx =  ( y, u, \lambda )$ of \eref{eq:model-ocp}
  satisfies $y \in H^{k+1}(\Omega)$ and $\lambda \in H^{k+1}(\Omega)$.
  If $\tau_e^{\sf s}$ satisfies \eref{eq:tau-ineq} and \eref{eq:tau-def},
  then there exists $\kappa > 0$ such that \eref{eq:abs-stab} holds.
\end{lemma}
\begin{proof}
  We apply Lemma \tref{eq:K-stab}.
  By Lemma \tref{l:ah-stab-est} $A_h$ is invertible and satisfies
  $\| A_h^{-1} \|_{{\cal L}(Y_h^*,Y_h)} \le 2$.
  It is easy to see that there exists $c > 0$ such that
  $\| B_h \|_{{\cal L}(U_h,Y_h^*)} \le c$ and 
  $\| H^{yy}_h \|_{{\cal L}(Y_h,Y_h^*)} \le c$ for all $h$.
  Finally, since $B_h^* (A_h^*)^{-1} H^{yy}_h A_h^{-1} B_h$ is postive
  semi definite, 
  \[
        \langle \widehat{H}_h u_h, u_h \rangle_{U_h^*\times U_h}
      \ge \langle H^{uu}_h u_h, u_h \rangle_{U_h^*\times U_h}
       = \omega \| u_h  \|_0^2,
  \]
  which implies $\| \widehat{H}_h^{-1} \|_{{\cal L}(U_h^*,U_h)} \le \omega^{-1}$.
\end{proof}

Now we turn to the consistency, i.e., we want to find an
estimate for $\| \br_h - \BK_h \BR_h(\bx) \|_h$ in our abstract
error estimate \eref{eq:abs-error}.
Let $\bz_h =  ( \psi_h, w_h, v_h ) \in \CX_h$.
The optimality conditions \eref{eq:model-disc-opt-cond} of the
discretized problem imply that
\begin{equation}   \label{eq:disc-opt-consist00}
     \langle \br_h - \BK_h \BR_h(\bx), \bz_h \rangle_{\CX_h^*, \CX_h}
   = \left(\begin{array}{c}
     a_h^{\sf s}(\psi_h,\lambda^I) + \langle y^I - \hat{y}, \psi_h \rangle \\
     b_h^{\sf s}(w_h,\lambda^I) + \omega \langle P u, w_h \rangle \\
     a_h^{\sf s}(y^I,v_h) + b_h^{\sf s}(P u,v_h)
                       - \langle f, v_h \rangle^{\sf s}_{h}
                       - \langle g, v_h \rangle_{\Gamma_n}
         \end{array} \right) .
\end{equation}
Since the solution $\bx = ( y,  u, \lambda )$
of \eref{eq:model-ocp} satisfies \eref{eq:model-opt-cond}
and since $y$ satisfies \eref{eq:intro-state}
on each $T_e \in \CT_h$, we find that
\begin{eqnarray*}
       a(\psi_h,\lambda)
  &=& -  \langle y - \hat{y}, \psi_h \rangle, \\
     b(w_h,\lambda) + \omega \langle u, w_h \rangle
  &=&  0, \\
    a_h^{\sf s}(y,v_h) + b_h^{\sf s}(u,v_h)
  &=&  \langle f, v_h \rangle^{\sf s}_{h} + \langle g, v_h \rangle_{\Gamma_n}
\end{eqnarray*}
for all $\psi_h, v_h \in V_h$ and $w_h \in U_h$.
With \eref{eq:Proj-def} this implies
\begin{equation}   \label{eq:disc-opt-consist0}
     \langle \br_h - \BK_h \BR_h(\bx), \bz_h \rangle_{\CX_h^*, \CX_h}
   = \left(\begin{array}{c}
     \sum_{T_e \in \CT_h} \tau_e^{\sf s}
          \langle - \epsilon \Delta \psi_h + \bc \cdot \nabla \psi_h + r \psi_h,
                 \bc \cdot \nabla \lambda^I \rangle_{0,T_e}
             + \langle y^I - y, \psi_h \rangle \\[1ex]
     \langle \lambda - \lambda^I, v_h \rangle
     - \sum_{T_e \in \CT_h} \tau_e^{\sf s}
                  \langle w_h, \bc \cdot \nabla \lambda^I \rangle_{0,T_e} \\[1ex]
     a_h^{\sf s}(y^I-y,v_h)
     - \sum_{T_e \in \CT_h} \tau_e^{\sf s}
                  \langle Pu -u, \bc \cdot \nabla v_h \rangle_{0,T_e}
         \end{array} \right).
\end{equation}

\begin{lemma}  \label{l:disc-opt-consist}
  Let $k, m \ge 1$ and suppose
  the solution $\bx =  ( y, u, \lambda )$ of \eref{eq:model-ocp}
  satisfies $y, \lambda \in H^{k+1}(\Omega)$ and $u \in H^{m+1}(\Omega)$.
  If $\tau_e^{\sf s}$ satisfies \eref{eq:tau-ineq} and \eref{eq:tau-def},
  then
  \begin{equation}   \label{eq:disc-opt-consist-1}
        \| \br_h - \BK_h \BR_h(\bx) \|_{\CX'}
   \le  C \left\{ \begin{array}{ll}
         ( \epsilon^{1/2} + h^{1/2} ) h^k  | y |_{k+1}
           + h^{m+2}\epsilon^{-1/2}  | u |_{m+1}             \\
           + h \epsilon^{-1/2} \| \nabla \lambda^I \|_0
           + h^{k+1}  ( | y |_{k+1} + | \lambda|_{k+1} ),& \Pe_e \le 1, \\[2ex]
         ( \epsilon^{1/2} + h^{1/2} ) h^k  | y |_{k+1}
           + h^{m+3/2}  | u |_{m+1}                          \\
           + ( \epsilon^{1/2} + h^{1/2} ) \| \nabla \lambda^I \|_0
           + h^{k+1}  ( | y |_{k+1} + | \lambda|_{k+1} ), & \Pe_e > 1.
        \end{array} \right. 
  \end{equation}
\end{lemma}
\begin{proof}
The terms in \eref{eq:disc-opt-consist0} can be estimated as follows.
Using the H\"older inequality and \eref{eq:inv-est} gives
\begin{eqnarray} \label{eq:disc-opt-consist-est1}
   \left| \sum_{T_e \in \CT_h} \tau_e^{\sf s}
    \langle - \epsilon \Delta \psi_h,\bc \cdot \nabla \lambda^I \rangle_{0,T_e}
   \right|
 &\le& \sum_{T_e \in \CT_h} \tau_e^{\sf s} \epsilon
                          \| \Delta \psi_h \|_{0,T_e}
                          \| \bc \|_{0,\infty,T_e}
                          \| \nabla \lambda^I \|_{0,T_e}  \nonumber \\
 &\le& C \left( \sum_{T_e \in \CT_h} \epsilon (\tau_e^{\sf s})^2/h_e^2
                                      \| \nabla \lambda^I \|_{0,T_e}^2
         \right)^{1/2}
         \left( \sum_{T_e \in \CT_h} \epsilon \| \nabla \psi_h \|_{0,T_e}^2
         \right)^{1/2} .
\end{eqnarray}
Standard estimates give
\begin{eqnarray}   \label{eq:disc-opt-consist-est3}
   \left|  \sum_{T_e \in \CT_h} \tau_e^{\sf s}
    \langle \bc \cdot \nabla \psi_h,\bc \cdot \nabla \lambda^I \rangle_{0,T_e}
   \right|
 &\le& \sum_{T_e \in \CT_h} \tau_e^{\sf s}
                          \| \bc \cdot \nabla \psi_h \|_{0,T_e}
                          \| \bc \|_{0,\infty,T_e}
                          \| \nabla \lambda^I \|_{0,T_e}  \nonumber \\
 &\le& C \left( \sum_{T_e \in \CT_h}
                \tau_e^{\sf s} \|   \nabla \lambda^I  \|_{0,T_e}^2
         \right)^{1/2}
       \left( \sum_{T_e \in \CT_h}
                \tau_e^{\sf s} \|  \bc \cdot \nabla \psi_h  \|_{0,T_e}^2
         \right)^{1/2}
\end{eqnarray}
and
\begin{eqnarray}    \label{eq:disc-opt-consist-est4}
   \left| \sum_{T_e \in \CT_h} \tau_e^{\sf s}
    \langle r \psi_h,\bc \cdot \nabla \lambda^I \rangle_{0,T_e} \right|
 &\le& \sum_{T_e \in \CT_h} \tau_e^{\sf s}
                          \| r \|_{0,\infty,T_e}
                          \| \psi_h \|_{0,T_e}
                          \| \bc \|_{0,\infty,T_e}
                          \| \nabla \lambda^I \|_{0,T_e}  \nonumber  \\
 &\le& C \left( \sum_{T_e \in \CT_h}
                (\tau_e^{\sf s})^2 \|   \nabla \lambda^I  \|_{0,T_e}^2
         \right)^{1/2}  \| \psi_h  \|_{0} .
\end{eqnarray}
The estimate \eref{eq:int-error} implies
\begin{eqnarray}    \label{eq:disc-opt-consist-est5}
   \big| \langle y^I - y, \psi_h \rangle \big|
 &\le& \mu_{\rm int} h^{k+1} |y|_{k+1} \| \psi_h \|_0 .
\end{eqnarray}
Combining \eref{eq:disc-opt-consist-est1}-\eref{eq:disc-opt-consist-est5}
gives
\begin{eqnarray}   \label{eq:adj-est-otd-1}
 \lefteqn{ \left| \sum_{T_e \in \CT_h} \tau_e^{\sf s}
          \langle - \epsilon \Delta \psi_h + \bc \cdot \nabla \psi_h + r \psi_h,
                    \bc \cdot \nabla \lambda^I \rangle_{0,T_e}
        + \langle y^I - y, \psi_h \rangle \right| } \nonumber \\
 &\le& C \left[ \Big( \sum_{T_e \in \CT_h}
                   (\epsilon (\tau_e^{\sf s})^2/h_e^2
                    + \tau_e^{\sf s} + (\tau_e^{\sf s})^2 )
                   \| \nabla \lambda^I \|_{0,T_e}^2
                \Big)^{1/2}
           + h^{k+1} |y|_{k+1} \right] \;   \| \psi_h  \|_{SD} .
\end{eqnarray}

Analogously to \eref{eq:disc-opt-consist-est4},
\eref{eq:disc-opt-consist-est5} we obtain
\begin{equation}  \label{eq:grad-est-otd}
    \left| \sum_{T_e \in \CT_h} \tau_e^{\sf s}
               \langle w_h, \bc \cdot \nabla \lambda^I \rangle_{0,T_e} \right|
  + \Big| \langle \lambda^I - \lambda, w_h \rangle \Big|
 \le C \left[ 
       \Big( \sum_{T_e \in \CT_h}
                   (\tau_e^{\sf s})^2  \| \nabla \lambda^I \|_{0,T_e}^2
       \Big)^{1/2}
       + h^{k+1} |\lambda|_{k+1} \right] \;
      \| w_h  \|_{0} 
\end{equation}

Using the esimates in Lemma~\tref{eq:est} we find that
\begin{equation}  \label{eq:state-est-otd1}
  | a_h^{\sf s}(y^I - y,v_h) |
  \le C h^k \left( \sum_{T_e \in \CT_h}
                    ( \epsilon + \tau_e^{\sf s}
                       + h_e^2/\tau_e^{\sf s} + h_e^2 )
                      | y |_{k+1,T_e}^2 \right)^{1/2}   \| v_h \|_{SD}.
\end{equation}
Finally, using standard estimates and \eref{eq:Pu-est} we obtain
\begin{eqnarray} \label{eq:state-est-otd2}
  \left| \sum_{T_e \in \CT_h} \tau_e^{\sf s}
                  \langle P u - u, \bc \cdot \nabla v_h \rangle_{0,T_e} \right|
 &\le& \left( \sum_{T_e \in \CT_h} 
             \tau_e^{\sf s} \|  P u - u \|^2_{0,T_e} \right)^{1/2}
        \left( \sum_{T_e \in \CT_h} 
             \tau_e^{\sf s} \| \bc \cdot \nabla v_h \|^2_{0,T_e} \right)^{1/2}
                     \nonumber \\
 &\le& C (\max_{T_e \in \CT_h} \tau_e^{\sf s})^{1/2} h^{m+1} | u |_{m+1} 
       \| v_h \|_{SD}.
\end{eqnarray}

The desired results now follows from \eref{eq:disc-opt-consist0},
the estimates \eref{eq:adj-est-otd-1}-\eref{eq:state-est-otd2},
the fact that 
\[
    \epsilon + \tau_e^{\sf s} + h_e^2/\tau_e^{\sf s} + h_e^2
  \le C ( \epsilon + h_e ) ,   \qquad
    \epsilon (\tau_e^{\sf s})^2/h_e^2 + \tau_e^{\sf s}
    + (\tau_e^{\sf s})^2
\le C \left\{ \begin{array}{ll}
       h_e^2/\epsilon , & \Pe_e \le 1, \\
       \epsilon + h_e,          & \Pe_e > 1,
       \end{array} \right.
\]
for $\tau_e^{\sf s}$ satisfying \eref{eq:tau-def}, and
\begin{equation}  \label{eq:norm-r-Kx}
   \| \br_h - \BK_h \BR_h(\bx) \|_{\CX_h^*}
 = \sup_{\bz_h \not= 0}
     \frac{\langle \br_h - \BK_h \BR_h(\bx), \bz_h \rangle_{\CX_h^*, \CX_h}}
          {\| \bz_h \|_{\CX_h}}.
\end{equation}
\end{proof}

If we apply the abstract error estimate \eref{eq:abs-error}
and combine the results in Lemmas \tref{l:x-Rx}, \tref{l:disc-opt-stab}
and \tref{l:disc-opt-consist} we obtain the following error estimate.
\begin{theorem}  \label{th:disc-opt-err}
  Let $k, m \ge 1$ and suppose the solution 
  $( y, u, \lambda )$ of \eref{eq:model-ocp}
  satisfies $y, \lambda \in H^{k+1}(\Omega)$, $u \in H^{m+1}(\Omega)$.
  If $\tau_e^{\sf s}$ satisfies \eref{eq:tau-ineq} and \eref{eq:tau-def}
  and if $\tau_e^{\sf a}$ satisfies \eref{eq:tau-ineq-a} and \eref{eq:tau-def},
  then the error between 
  the solution $( y, u, \lambda )$ of \eref{eq:model-ocp} and
  the solution $( y_h, u_h, \lambda_h )$ of the discretized problem
  \eref{eq:model-ocp-disc} obeys
  \begin{eqnarray}   \label{eq:disc-opt-err}
   \lefteqn{ \| y_h - y \|_{SD} + \| u_h - u \|_0 
             + \| \lambda_h - \lambda \|_{SD} }     \nonumber \\[2ex]
  &\le& C \left\{ \begin{array}{ll}
         ( \epsilon^{1/2} + h^{1/2} ) h^k  | y |_{k+1}
           + h^{m+2}\epsilon^{-1/2}  | u |_{m+1}             \\
           + h \epsilon^{-1/2} \| \nabla \lambda^I \|_0
           + h^{k+1}  ( | y |_{k+1} + | \lambda|_{k+1} ),& \Pe_e \le 1, \\[2ex]
         ( \epsilon^{1/2} + h^{1/2} ) h^k  | y |_{k+1}
           + h^{m+3/2}  | u |_{m+1}                          \\
           + ( \epsilon^{1/2} + h^{1/2} ) \| \nabla \lambda^I \|_0
           + h^{k+1}  ( | y |_{k+1} + | \lambda|_{k+1} ), & \Pe_e > 1.
        \end{array} \right. 
  \end{eqnarray}
\end{theorem}

Theorem~\tref{th:disc-opt-err} gives an estimate for the 
states, controls and adjoints combined. Our numerical results in 
Section~\tref{sec:numer} show that this error estimate
is often too conservative for the states and the controls.
The reason for this is that while the error 
$\lambda - \lambda_h$ in the $\| \cdot \|_{SD}$ norm
behaves as in \eref{eq:disc-opt-err}, the error
$\lambda - \lambda_h$ measured in the $L^2$-norm is often much 
smaller. 
Because of the optimality conditions \eref{eq:intro-grad}
and \eref{eq:model-disc-grad} this tends to imply a smaller
error $\| u - u_h \|_0$ in the control than the one suggested 
by \eref{eq:disc-opt-err}. $L^2$ and $L^\infty$ error estimates for the SUPG 
method are discussed, e.g., in \cite{GZhou_1997a,GZhou_RRannacher_1996a}. 
See also the overview in \cite[Sec.~3.2.1]{HGRoos_MStynes_LTobiska_1996}.
However, $L^2$ estimates for the error $\lambda - \lambda_h$
in the optimal control context and $L^2$ estimates for the error 
$u - u_h$ have not yet been established. 

If the error $u - u_h$ in the control is smaller than 
the upper bound established in \eref{eq:disc-opt-err},
we can obtain an improved estimate for the error in the
states. This is stated in the next theorem.

\begin{theorem}  \label{th:disc-opt-state}
  Let $k \ge 1$ and suppose the solution $y$ of \eref{eq:model-state-weak}
  satisfies $y \in H^{k+1}(\Omega)$.
  Furthermore, let $y_h$ solve \eref{eq:model-disc-state}.
  If $\tau_e^{\sf s}$ satisfies \eref{eq:tau-ineq} and \eref{eq:tau-def},
  then there exists $C>0$ such that
  \begin{equation}   \label{eq:disc-opt-state-err}
        \| y - y_h \|_{SD}
    \le  C \Big( h^k ( \epsilon^{1/2} +  h^{1/2}) | y |_{k+1}
                 + \| u_h-u \|_0 \Big)   \qquad \forall h.
  \end{equation}
\end{theorem}
\begin{proof}
  Let $\widetilde{y}_h \in Y_h$ be the solution of
  \begin{equation}  \label{eq:model-state-disc1}
       a_h^{\sf s}(\widetilde{y}_h,v_h) + b_h^{\sf s}(u,v_h)
    = \langle f, v_h \rangle_h^{\sf s} + \langle g, v_h \rangle_{\Gamma_n}
   \quad \forall v_h \in V_h.
  \end{equation}
  Theorem~\tref{th:supg-err} implies
  \[
       \| y - \widetilde{y}_h \|_{SD}
       \le   C h^k ( \epsilon^{1/2} +  h^{1/2}) | y |_{k+1}.
  \]
  To estimate $y_h -\widetilde{y}_h$ we
  subtract \eref{eq:model-state-disc1} from \eref{eq:model-state-disc},
  \[
   a_h^{\sf s}(y_h -\widetilde{y}_h,v_h) + b_h^{\sf s}(u_h-u,v_h)
   = 0 \quad \forall v_h \in V_h,
  \]
  set $v_h = y_h -\widetilde{y}_h$ and apply Lemma~\tref{eq:tau-ineq}
  to obtain
  \[
     \half \| y_h -\widetilde{y}_h \|_{SD}
     \le a_h^{\sf s}(y_h -\widetilde{y}_h,y_h -\widetilde{y}_h) 
      = -b_h^{\sf s}(u_h-u,y_h -\widetilde{y}_h)
     \le \| u_h-u \|_0 \| y_h -\widetilde{y}_h \|_{SD}.
  \]
  This implies the desired estimate. 
\end{proof}

The next result indicates that the estimate \eref{eq:disc-opt-err}
for the error in the adjoints cannot be improved. Even if
we solve the discrete adjoint \eref{eq:model-disc-adj} with $y_h$ 
replaced by the optimal state $y$, i.e., the solution of \eref{eq:model-ocp},
we obtain an error estimate comparable to \eref{eq:disc-opt-err}.

\begin{theorem}  \label{th:disc-opt-adj}
  Let $y$ be the optimal state, i.e., the solution of \eref{eq:model-ocp}.
  Let $k \ge 1$ and suppose the solution $\lambda$ of \eref{eq:model-adj-weak}
  satisfies $\lambda \in H^{k+1}(\Omega)$.
  Furthermore, let $\lambda_h$ solve
  \begin{equation}  \label{eq:model-disc-adj-exy}
     a_h^{\sf s}(\psi_h,\lambda_h)
     = - \langle y - \hat{y}, \psi_h \rangle
        \quad \forall \psi_h \in V_h.
  \end{equation}
  If $\tau_e^{\sf s}$ satisfies \eref{eq:tau-ineq} and \eref{eq:tau-def},
  then
  \begin{equation}   \label{eq:disc-opt-adj-err}
        \| \lambda - \lambda_h \|_{SD}
    \le  C \left\{ \begin{array}{ll}
            h \epsilon^{-1/2} \| \nabla \lambda^I \|_0,& \Pe_e \le 1, \\[2ex]
        ( \epsilon^{1/2} + h^{1/2} ) \| \nabla \lambda^I \|_0, & \Pe_e > 1.
        \end{array} \right.
  \end{equation}
\end{theorem}
\begin{proof}
  This result follows from the stability result \eref{eq:ah-stable}
  and the consistency estimates 
  \eref{eq:disc-opt-consist-est1}-\eref{eq:disc-opt-consist-est4}.
  All other steps in the proof of this result are analogous to those
  in the proof of Theorem~\tref{th:supg-err} given, e.g., 
  \cite[Thm.~3.30]{HGRoos_MStynes_LTobiska_1996} or
  \cite[Thm.~9.3]{PKnabner_LAngermann_2000}.
\end{proof}

%%%%%%%%%%%%%%%%%%%%%%%%%%%%%%%%%%%%%%%%%%%%%%%%%%%%%%%%%%%%%%%%%%%%%%%%%%%

\subsection{Optimize-Then-Discretize}

In the optimize-then-discretize approach the
discrete equation \eref{eq:Khxh=rh} corresponds to 
\eref{eq:model-opt-cond-disc}. The components of
$\BK_h$ in \eref{eq:Kh-def} are given by
\begin{equation}   \label{eq:Kh-def-opt-disc}
\begin{array}{c}
        \langle H^{yy}_h y_h, \psi_h \rangle_{\Lambda_h^* \times  \Lambda_h}
      = \langle y_h, \psi_h \rangle_h^{\sf a}
      = \langle y_h , \psi_h \rangle
       + \sum_{T_e \in \CT_h} \tau_e^{\sf a}
                   \langle y_h, -\bc \cdot \psi_h \rangle_{T_e},  \\[1ex]
        \langle H^{uu}_h u_h, w_h \rangle_{U_h^* \times U_h}
      = \omega \langle u_h, w_h \rangle, \quad
      H^{uy}_h = H^{yu}_h = 0,      \\[1ex]
      \langle A_h y_h, v_h \rangle_{V_h^* \times V_h}
      = a_h^{\sf s}(y_h, v_h), \quad
      \langle B_h u_h, v_h \rangle_{V_h^*\times V_h}
      = b_h^{\sf s}(u_h, v_h),  \\[1ex]
      \langle \widetilde{A}_h \psi_h, 
                              \lambda_h \rangle_{\Lambda_h^* \times \Lambda_h}
      = a_h^{\sf a}(\psi_h, \lambda_h), \quad
    \langle \widetilde{B}_h u_h, \lambda_h \rangle_{\Lambda_h^*\times \Lambda_h}
      = b_h(u_h, \lambda_h).
\end{array}
\end{equation}

As we have pointed out at the end of Section \tref{sec:opt-then-disc},
the discretization of the optimality system leads to a non-selfadjoint
system. This makes the derivation of a stability result \eref{eq:abs-stab}
more complicated than in the discetize-then-optimize approach.
On the other hand, derivation of a consistency estimate is just a simple
application of the standard SUPG consistency estimates 
reviewed in Section \tref{sec:supg-state-error}.

We first derive a stability result.
The next  lemma collects some preliminary results on the
solution of the state and adjoint equations as well as
their approximations computed using the SUPG method.

\begin{lemma}  \label{l:stab-est-temp}
  {\rm i.}
  Let $u \in L^2(\Omega)$. Suppose that the solution $z(u) \in V$ of
  \begin{equation}   \label{eq:z-def00}
     a(z, v) = b(u, v)  \quad \forall v \in V
  \end{equation}
  satisfies $z(u) \in H^2(\Omega)$ and that there exists
  $C > 0$ independent of $u$ such that
  \begin{equation}   \label{eq:z-def0}
      \| z(u) \|_2 \le C \| u \|_0.
  \end{equation}
  If $z_h(u_h)$ and $\widetilde{z}_h(u_h)$ solve
  \begin{equation}   \label{eq:z-def1}
     a_h^{\sf s}(z_h, v_h) = b_h^{\sf s}(u, v_h)  \quad \forall v_h \in V_h
  \end{equation}
  and
  \begin{equation}   \label{eq:z-def2}
     a_h^{\sf a}(\widetilde{z}_h, \psi_h) = b(u, \psi_h)  
           \quad \forall \psi_h \in \Lambda_h,
  \end{equation}
  respectively, and 
  if $\tau_e^{\sf s}$ satisfies \eref{eq:tau-ineq} and \eref{eq:tau-def},
  then there exists $C > 0$ independent of $u_h$ such that
  \begin{equation}   \label{eq:z-est1}
       \| z_h(u) \|_{SD} \le C \| u \|_0, \quad
       \| \widetilde{z}_h(u_h) \|_{SD} \le C \| u \|_0,
  \end{equation}
  \begin{equation}   \label{eq:z-est2}
       \| z_h(u) - z(u) \|_{SD} 
       \le C h ( \epsilon^{1/2} +  h^{1/2}) \| u \|_0.
  \end{equation}

  {\rm ii.}
  Let $z \in L^2(\Omega)$. Suppose that the solution $\mu(z) \in V$ of
  \begin{equation}   \label{eq:mu-def00}
       a(\psi, \mu)   = \langle z, \psi \rangle   \quad \forall \psi \in V,
  \end{equation}
  satisfies $\mu(z) \in H^2(\Omega)$ and that there exists
  $C > 0$ independent of $z$ such that
  \begin{equation}   \label{eq:mu-def0}
      \| \mu(z) \|_2 \le C \| z \|_0.
  \end{equation}
  If $\mu_h(z) \in \Lambda_h$ solves
  \begin{equation}   \label{eq:mu-def1}
       a_h^{\sf a}(\psi_h, \mu_h) 
       = \langle z, \psi_h \rangle
          +  \sum_{T_e \in \CT_h} \tau_e^{\sf a}
                   \langle z, -\bc \cdot \nabla \psi_h \rangle_{T_e} 
          \quad \forall \psi_h \in \Lambda_h
  \end{equation}
  and if $\tau_e^{\sf a}$ satisfies 
  \begin{equation}   \label{eq:tau-ineq-a}
    0 < \tau_e^{\sf a}
     \le \min \left\{ \frac{ h_e^2 }{ \epsilon \mu_{\rm inv}^2 },
                     \frac{ r_0 }{ \| r - \nabla \cdot \bc\|_{0,\infty,T_e} }
              \right\} 
  \end{equation}
  and \eref{eq:tau-def} with $\tau_e^{\sf s}$ replaced by $\tau_e^{\sf a}$,
  then
  \begin{equation}   \label{eq:mu-est1}
       \| \mu_h(z) - \mu(z) \|_{SD}
    \le C h ( \epsilon^{1/2} +  h^{1/2}) \| z \|_0.
  \end{equation}
\end{lemma}
\begin{proof}
  The inequalities \eref{eq:z-est1} follow from the SUPG stability
  estimate Lemma \tref{l:ah-stab-est},
  inequalities \eref{eq:z-est2}, \eref{eq:mu-est1} follow from the 
  SUPG convergence theory,
  cf. Theorem~\tref{th:supg-err}.
\end{proof}

\begin{lemma}  \label{l:opt-disc-stab}
  Let the assumptions of Lemma~\tref{l:stab-est-temp} be satisfied.
  Let $k, \ell, m \ge 1$ and suppose the solution 
  $\bx =  ( y, u, \lambda )$ of \eref{eq:model-ocp}
  satisfies $y \in H^{k+1}(\Omega)$, $u \in H^{m+1}(\Omega)$,
  $\lambda \in H^{\ell+1}(\Omega)$.
  If $\tau_e^{\sf s}$ satisfies \eref{eq:tau-ineq} and \eref{eq:tau-def}
  and if $\tau_e^{\sf a}$ satisfies \eref{eq:tau-ineq-a}
  and \eref{eq:tau-def} with $\tau_e^{\sf s}$ replaced by $\tau_e^{\sf a}$,
  then there exist $\bar{h} > 0$ and 
  $\kappa > 0$ such that $\BK_h$ is invertible for all $h \ge \bar{h}$
  and $\| \BK_h^{-1} \|_h \le \kappa$ for all $h \ge \bar{h}$.
\end{lemma}
\begin{proof}
  We apply Lemma \tref{eq:K-stab}.
  By Lemma \tref{l:ah-stab-est} $A_h$ and $\widetilde{A}_h$ 
  are invertible and  satisfy 
  $\| A_h^{-1} \|_{{\cal L}(Y_h^*,Y_h)} \le 2$,
  $\| \widetilde{A}_h^{-1} \|_{{\cal L}(\Lambda_h^*,\Lambda_h)} \le 2$.
  It is straightforward to verify that there exists $c > 0$ such that
  $\| B_h \|_{{\cal L}(U_h,Y_h^*)} \le c$,
  $\| \widetilde{B}_h \|_{{\cal L}(U_h,\Lambda_h^*)} \le c$ and 
  $\| H^{yy}_h \|_{{\cal L}(Y_h,\Lambda_h^*)} \le c$ for all $h$.
  
  The proof of uniform boundedness of $\widehat{H}_h^{-1}$ is a little
  bit more involved than in Lemma  \tref{l:disc-opt-stab}
  because $\widetilde{B}_h^* (\widetilde{A}_h^*)^{-1} H^{yy}_h A_h^{-1} B_h$
  is not symmetric.
 
  By Definition \tref{eq:Kh-def-opt-disc}
  of $A_h$, $\widetilde{A}_h$, $B_h$ and
  $\widetilde{B}_h$ the vectors $z_h = A_h^{-1} B_h u_h$ and
  $\widetilde{z}_h = \widetilde{A}_h^{-1} \widetilde{B}_h u_h$
  solve \eref{eq:z-def1} and \eref{eq:z-def2}, respectively,
  with $u = u_h$.
  Let $z$ be the solution of \eref{eq:z-def00} with $u = u_h$.
  From the Definition \tref{eq:red-H-def} of $\widehat{H}_h$
  we obtain
  \begin{eqnarray}  \label{eq:Hhat-est1}
     \lefteqn{\langle \widehat{H}_h u_h, u_h \rangle_{U_h^* \times U_h}}
                        \nonumber \\
      &=& \langle H^{uu}_h u_h, u_h \rangle_{U_h^*\times U_h}
          + \langle \widetilde{B}_h^* (\widetilde{A}_h^*)^{-1} 
                  H^{yy}_h A_h^{-1} B_h u_h, u_h \rangle_{U_h^*, U_h} 
                          \nonumber   \\
      &=& \langle H^{uu}_h u_h, u_h \rangle_{U_h^* \times U_h}
          + \langle H^{yy}_h A_h^{-1} B_h u_h, 
                  \widetilde{A}_h^{-1} \widetilde{B}_h u_h
                       \rangle_{\Lambda_h^* \times \Lambda_h}  
                        \nonumber \\
      &=& \omega \| u_h \|_0^2
           + \langle H^{yy}_h  z, z \rangle_{\Lambda_h^* \times \Lambda_h} 
           + \langle H^{yy}_h (z_h - z),
                    \widetilde{z}_h \rangle_{\Lambda_h^* \times \Lambda_h} 
           + \langle H^{yy}_h z,
                    \widetilde{z}_h -z\rangle_{\Lambda_h^* \times \Lambda_h}.
  \end{eqnarray}
  Applying the definition \eref{eq:Kh-def-opt-disc} of $H^{yy}_h$
  and \eref{eq:z-def0}-\eref{eq:z-est2}, we obtain
  \begin{eqnarray}  \label{eq:Hhat-est2}
      \langle H^{yy}_h  z, z \rangle_{\Lambda_h^* \times \Lambda_h} 
    &=& \|  z \|_0^2
        +  \sum_{T_e \in \CT_h} \tau_e^{\sf a}
                   \langle z, -\bc \cdot \nabla z \rangle_{T_e} \nonumber \\
    &\ge& \|  z \|_0^2
          -  \max_{T_e \in \CT_h} \tau_e^{\sf a} \;
                  \| \bc \|_{0,\infty} \| z \|_0 \| z \|_1    \nonumber \\
    &\ge& \|  z \|_0^2
          -  C \max_{T_e \in \CT_h} \tau_e^{\sf a} \;
                  \| \bc \|_{0,\infty} \| u_h \|_0^2
  \end{eqnarray}
  and
  \begin{eqnarray} \label{eq:Hhat-est3}
      \langle H^{yy}_h  (z_h - z),
              \widetilde{z}_h \rangle_{\Lambda_h^* \times \Lambda_h} 
   &\ge& - \| z_h - z \|_0 \| \widetilde{z}_h \|_0
        - \max_{T_e \in \CT_h} \tau_e^{\sf a} \;
              \| \bc \|_{0,\infty} \| z_h - z \|_0 \| \widetilde{z}_h \|_1
               \nonumber \\
   &\ge& - C h ( \epsilon^{1/2} +  h^{1/2})
           \left( 1+ \| \bc \|_{0,\infty}
                     \max_{T_e \in \CT_h} \tau_e^{\sf a} \right)
           \| u_h \|_0^2.
  \end{eqnarray}
  To estimate the last term in \eref{eq:Hhat-est1} we set
  we set $\widetilde{\mu} = (A^*)^{-1} H^{yy}_h z$ and 
  $\mu_h = (\widetilde{A}^*)^{-1} H^{yy}_h z$.
  The identity $A^*\widetilde{\mu} = H^{yy}_h z$  and the definitions 
  of $A$ and $H^{yy}_h$ imply that
  \[
       a(\psi, \widetilde{\mu}) 
    = \langle z, \psi \rangle
        +  \sum_{T_e \in \CT_h} \tau_e^{\sf a}
                   \langle z, -\bc \cdot \nabla \psi \rangle_{T_e} 
          \quad \forall \psi \in V.
  \]
  Similarly, $\mu_h$ solves \eref{eq:mu-def1}.
  If $\mu$ solves \eref{eq:mu-def00}, then the Lipschitz continuity of the
  solution of the adjoint equation with respect to perturbations in the
  right hand side implies
  \[
        \| \mu - \widetilde{\mu} \|_1
    \le C \max_{T_e \in \CT_h} \tau_e^{\sf a} \; 
          \| \bc \|_{1,\infty} \| z \|_1
    \le C \max_{T_e \in \CT_h} \tau_e^{\sf a} \; 
          \| \bc \|_{1,\infty} \| u_h \|_0.
  \]
  (To obtain the last inequality, recall that $z$ solves \eref{eq:z-def00} 
  with $u = u_h$ and that \eref{eq:z-def0} holds with $u = u_h$.)
  If $\mu$ solves \eref{eq:mu-def00}, then \eref{eq:mu-est1} and
  \eref{eq:z-def0} imply
  \[   
           \| \mu - \mu_h \|_{SD}
        \le C h ( \epsilon^{1/2} +  h^{1/2}) \| z \|_0
        \le C h ( \epsilon^{1/2} +  h^{1/2}) \| u_h \|_0.
  \]
  Hence
  \begin{eqnarray} \label{eq:Hhat-est4}
      \langle H^{yy}_h z,
             \widetilde{z}_h -z \rangle_{\Lambda_h^* \times \Lambda_h} 
    &=& \langle (A^*)^{-1} H^{yy}_h z - (\widetilde{A}^*)^{-1} H^{yy}_h z,
               \widetilde{B}_h u_h \rangle_{\Lambda_h \times \Lambda^*_h}  
                              \nonumber \\
    &=& \langle \widetilde{\mu} - \mu_h,
               \widetilde{B}_h u_h \rangle_{\Lambda_h \times \Lambda^*_h}
    \ge \| \widetilde{\mu} - \mu_h \|_{SD} 
         \| \widetilde{B}_h u_h \|_{\Lambda^*_h}
                              \nonumber \\
    &\ge& C \Big( \max_{T_e \in \CT_h} \tau_e^{\sf a} 
                  + h ( \epsilon^{1/2} +  h^{1/2})  \Big)
          \| u_h \|_0^2.   
  \end{eqnarray}

  Estimates \eref{eq:Hhat-est1}-\eref{eq:Hhat-est4} imply the
  existence of  $\bar{h} > 0$ such that
  $\langle \widehat{H}_h u_h, u_h \rangle_{U_h^* \times U_h}
   \ge \half \omega \| u_h \|_0^2$.
  This implies 
  $\| \widehat{H}_h^{-1} \|_{{\cal L}(U_h^*,U_h)} \le 2 \omega^{-1}$.
  The desired result now follows from Lemma~\tref{eq:K-stab}.
\end{proof}

Now we turn to the consistency, i.e., we want to find an
estimate for $\| \br_h - \BK_h \BR_h(\bx) \|$ in our abstract
error estimate \eref{eq:abs-error}.
The discretized optimality condition \eref{eq:model-opt-cond-disc}
imply that
\[
     \langle \br_h - \BK_h \BR_h(\bx), \bz \rangle_{\CX_h^*, \CX_h}
   = \left(\begin{array}{c}
     a_h^{\sf a}(\psi_h,\lambda^I)
   +  \langle y^I - \hat{y}, \psi_h \rangle_h^{\sf a}  \\
   b( w_h, \lambda^I) + \omega \langle P u, w_h \rangle \\
   a_h^{\sf s}(y^I,v_h) + b_h^{\sf s}(P u,v_h)
       - \langle f, v_h \rangle^{\sf s}_{h}
       - \langle g, v_h \rangle_{\Gamma_n}
         \end{array} \right).
\]
Since the solution $\bx = ( y,  u, \lambda )$
of \eref{eq:model-ocp} satisfies \eref{eq:model-opt-cond}
and that $y$ satisfies \eref{eq:intro-state}
on each $T_e \in \CT_h$, we have
\begin{eqnarray*}
       a_h^{\sf a}(\psi_h,\lambda)
  &=& -  \langle y - \hat{y}, \psi_h \rangle_h^{\sf a}, \\
     b(w_h,\lambda) + \omega \langle u, w_h \rangle
  &=&  0, \\
    a_h^{\sf s}(y,v_h) + b_h^{\sf s}(u,v_h)
  &=&  \langle f, v_h \rangle^{\sf s}_{h}+ \langle g, v_h \rangle_{\Gamma_n}
\end{eqnarray*}
for all $\psi_h, v_h \in V_h$ and $w_h \in U_h$.
With \eref{eq:Proj-def} this implies
\[
     \langle \br_h - \BK_h \BR_h(\bx), \bz \rangle_{\CX_h^*, \CX_h}
   = \left(\begin{array}{c}
     a_h^{\sf a}(\psi_h,\lambda^I - \lambda)
   +  \langle y^I - y, \psi_h \rangle_h^{\sf a}  \\
   b( w_h, \lambda^I - \lambda )  \\
   a_h^{\sf s}(y^I - y,v_h)
   - \sum_{T_e \in \CT_h} \tau_e^{\sf s}
                  \langle P u - u, \bc \cdot v_h \rangle_{0,T_e}
         \end{array} \right) .
\]

\begin{lemma}  \label{l:opt-disc-consist}
  Let $k, \ell, m \ge 1$ and suppose the solution 
  $\bx =  ( y, u, \lambda )$ of \eref{eq:model-ocp}
  satisfies $y \in H^{k+1}(\Omega)$, $u \in H^{m+1}(\Omega)$,
  $\lambda \in H^{\ell+1}(\Omega)$.
  If $\tau_e^{\sf s}$ satisfies \eref{eq:tau-ineq} and \eref{eq:tau-def}
  and if $\tau_e^{\sf a}$ satisfies \eref{eq:tau-ineq-a} and \eref{eq:tau-def},
  then
  \begin{equation}   \label{eq:opt-disc-consist}
        \| \br_h - \BK_h \BR_h(\bx) \|_{\CX}
   \le  C \left( ( \epsilon^{1/2} +  h^{1/2} ) 
                 ( h^k |y|_{k+1} + h^\ell |\lambda|_{\ell+1}) 
                 + h^{m+1} |u|_{m+1} \right) .
  \end{equation}
\end{lemma}
\begin{proof}
  Using the estimates in Lemma~\tref{eq:est} we find that
 \begin{eqnarray*}
   a_h^{\sf s}(y^I - y,v_h)
  &\le& C h^k \left( \sum_{T_e \in \CT_h}
                    ( \epsilon + \tau_e^{\sf s}
                       + h_e^2/\tau_e^{\sf s} + h_e^2 )
                      | y |_{k+1,T_e}^2 \right)^{1/2}   \| v_h \|_{SD}, \\
   a_h^{\sf a}(\psi_h,\lambda^I - \lambda)
  &\le& C h^\ell \left( \sum_{T_e \in \CT_h}
                    ( \epsilon + \tau_e^{\sf a}
                       + h_e^2/\tau_e^{\sf a} + h_e^2 )
                   | \lambda |_{\ell+1,T_e}^2 \right)^{1/2}  \| \psi_h \|_{SD}.
 \end{eqnarray*}
 Furthermore, \eref{eq:int-error} implies
 \[
   b( w_h, \lambda^I - \lambda )
   \le \mu_{\rm int}  h^{\ell+1}  |\lambda|_{\ell+1} \| w_h \|_0 .
 \]
 By  \eref{eq:state-est-otd2},
 \[
  \left| \sum_{T_e \in \CT_h} \tau_e^{\sf s}
                  \langle P u - u, \bc \cdot v_h \rangle_{0,T_e} \right|
 \le C \max_{T_e \in \CT_h} \tau_e^{\sf s} h^{m+1} | u |_{m+1} 
       \| v_h \|_{SD}.
 \]
 Finally
 \begin{eqnarray*}
   \langle y^I - y, \psi_h \rangle_h^{\sf a}
 &\le& \mu_{\rm int} h^{k+1} |y|_{k+1} \| \psi_h \|_0
       + \sum_{T_e \in \CT_h} \tau_e^{\sf a} \mu_{\rm int}
                            h^{k+1}_{T_e} |y|_{k+1,T_e}
                             \| \bc \cdot \nabla \psi_h \|_{0,T_e}\\
 &\le& C h^{k+1} |y|_{k+1} \| \psi_h \|_{SD}.
 \end{eqnarray*}

 The inequality \eref{eq:opt-disc-consist} is obtained by
 combining the above estimates, using that \eref{eq:tau-def}
 implies $\epsilon + \tau_e + h_e^2/\tau_e + h_e^2 \le C ( \epsilon + h_e )$
 and the identity \eref{eq:norm-r-Kx}.
\end{proof}

If we apply the abstract error estimate \eref{eq:abs-error}
and combine the results in Lemmas \tref{l:x-Rx}, \tref{l:opt-disc-stab}
and \tref{l:opt-disc-consist} we obtain the following error estimate.
\begin{theorem}  \label{th:opt-disc-err}
  Let $k, \ell, m \ge 1$ and suppose the solution 
  $( y, u, \lambda )$ of \eref{eq:model-ocp}
  satisfies $y \in H^{k+1}(\Omega)$, $u \in H^{m+1}(\Omega)$,
  $\lambda \in H^{\ell+1}(\Omega)$.
  If $\tau_e^{\sf s}$ satisfies \eref{eq:tau-ineq} and \eref{eq:tau-def}
  and if $\tau_e^{\sf a}$ satisfies \eref{eq:tau-ineq-a} and \eref{eq:tau-def},
  then the error between 
  the solution $( y, u, \lambda )$ of \eref{eq:model-ocp} and
  the solution $( y_h, u_h, \lambda_h )$ of the discretized 
  optimality conditions \eref{eq:model-opt-cond-disc}
  obeys
  \begin{eqnarray}   \label{eq:opt-disc-err}
   \lefteqn{ \| y_h - y \|_{SD} + \| u_h - u \|_0 
             + \| \lambda_h - \lambda \|_{SD} }     \nonumber \\[2ex]
  &\le& C \left( ( \epsilon^{1/2} +  h^{1/2} ) 
                 ( h^k |y|_{k+1} + h^\ell |\lambda|_{\ell+1}) 
                 + h^{m+1} |u|_{m+1} \right) .
  \end{eqnarray}
\end{theorem}

As in the case of Theorem~\tref{th:disc-opt-err},
Theorem~\tref{th:opt-disc-err} also gives an estimate for the 
states, controls and adjoints combined. This error estimate
is sometimes too conservative for the controls for the same
reasons sketched after Theorem~\tref{th:disc-opt-err}.
As in the discretize-then-optimize case,
$L^2$ estimates for the error $\lambda - \lambda_h$
in the optimal control context and $L^2$ estimates for the error 
$u - u_h$ have not yet been established.

If the error $u - u_h$ in the control is smaller than 
the upper bound established in \eref{eq:opt-disc-err},
we can obtain an improved estimate for the error in the
states and in the adjoints. This is stated in the next theorem.
Note that if a better estimate for the error $\lambda - \lambda_h$
can be obtained, this might also allow to further improve
the error $u - u_h$ in the control.

\begin{theorem}  \label{th:opt-disc-state-adj}
  Let $k, \ell \ge 1$ and suppose the solution
  $( y, u, \lambda )$ of \eref{eq:model-ocp}
  satisfies $y \in H^{k+1}(\Omega)$,
  $\lambda \in H^{\ell+1}(\Omega)$.
  Furthermore, let $y_h$, $\lambda_h$ solve \eref{eq:model-adj-disc}
  and \eref{eq:model-state-disc}, respectively.
  If $\tau_e^{\sf s}$ satisfies \eref{eq:tau-ineq} and \eref{eq:tau-def}
  and if $\tau_e^{\sf a}$ satisfies \eref{eq:tau-ineq-a} and \eref{eq:tau-def},
  then there exists $C>0$ such that
  \begin{eqnarray}   
       \label{eq:opt-disc-state-err}
        \| y - y_h \|_{SD}
    &\le&  C \Big( h^k ( \epsilon^{1/2} +  h^{1/2}) | y |_{k+1}
                 + \| u_h-u \|_0 \Big)   \qquad \forall h, \\
       \label{eq:opt-disc-adj-err}
        \| \lambda - \lambda_h \|_{SD}
    &\le&  C \Big( h^\ell ( \epsilon^{1/2} +  h^{1/2}) | \lambda |_{k+1}
                 + \| y_h-y \|_0 \Big)   \qquad \forall h.
  \end{eqnarray}
\end{theorem}
\begin{proof}
  The proof is analogous to the proof of Theorem~\tref{th:disc-opt-state}.
\end{proof}

A comparison of Theorems~\tref{th:disc-opt-err} and \tref{th:opt-disc-err}
indicates that the optimize-then-discretize approach leads to
asymptotically better approximate solutions of the optimal
control problem than the discretize-then-optimize approach, because 
the estimate \eref{eq:opt-disc-err} is dominated by the term 
$h \epsilon^{-1/2} \| \nabla \lambda^I \|_0$ and
$( \epsilon^{1/2} + h^{1/2} ) \| \nabla \lambda^I \|_0$, respecively.
The differences between the error bounds provided in 
Theorems~\tref{th:disc-opt-err} and \tref{th:opt-disc-err} is small
when piecewise linear polynomials are used for the discretization
states, adjoints and controls, i.e., if $k = \ell = m = 1$.
We also note that in the case of piecewise linear finite
elements the contributions
$\langle - \epsilon \Delta y_h, \bc \cdot \nabla v_h \rangle_{T_e}$ and 
$\langle - \epsilon \Delta \lambda_h,  -\bc \cdot \nabla \psi_h \rangle_{T_e}$ 
of the SUPG method to the bilinear forms \eref{eq:model-a-b}
and \eref{eq:model-a-adj} disappear and, hence, one source 
of difference between the discretize-then-optimize approach
and the optimize-then-discretize approach is eliminated.
This would not apply if reconstructions of second derivative
terms had been used \cite{KEJansen_SSCollis_CWhiting_FShakib_1999}.

The differences between the discretize-then-optimize approach
and the optimize-then-discretize approach
are the greater the larger $k , \ell$, i.e., the  higher 
the order of finite elements used for the states and the adjoints. 
Theorem~\tref{th:disc-opt-adj} and
Theorems~\tref{th:opt-disc-err}, \tref{th:opt-disc-state-adj}
show that there is a significant difference in quality of the 
adjoints computed by the discretize-then-optimize approach
and the optimize-then-discretize approach. The latter leads to
better adjoint approximations. This is confirmed by our
numerical results reported in Section~\tref{sec:numer}.
However, our numerical results also show that this large difference
in the quality of the adjoints does not necessarily implies a large
difference in the quality of the controls. Often the observed 
error in the controls computed by the discretize-then-optimize approach
and the optimize-then-discretize approach is very similar,
which by Theorem~\tref{th:disc-opt-state} and \tref{th:opt-disc-state-adj}
leads to very similar errors in the computed states for both approaches.

\section{Numerical Results}   \label{sec:numer}

%%%%%%%%%%%%%%%%%%%%%%%%%%%%%%%%%%%%%%%%%%%%%%%%%%%%%%%%%%%%%%%%%%%%%%%%
\subsection{Example 1}  \label{sec:ex1}

Our first example is a one dimensional control problem
on $\Omega = (0,1)$ with state equation 
\begin{equation}   \label{eq:ex1-state}
      - \epsilon y''(x) + y'(x) = f(x) + u(x) 
                      \mbox{ on } (0,1), \quad
      y(0) = y(1) = 0.
\end{equation}
The functions $f$ and $\widehat{y}$ are chosen so that the solution
of the optimal control problem is
\[
    y_{\rm ex}(x) =  x - \frac{ \exp((x-1)/\epsilon ) - \exp(-1/\epsilon) }
                         { 1 - \exp(-1/\epsilon) }, \quad
    \lambda_{\rm ex}(x)
     = \left(1-x - \frac{\exp(-x/\epsilon)-\exp(-1/\epsilon)}
                        {1 - \exp(-1/\epsilon)} \right)
\]  
and $u_{\rm ex} = \omega^{-1}  \lambda_{\rm ex}$.
This example is modeled after \cite[pp.~2,3]{HGRoos_MStynes_LTobiska_1996}.
We set $\epsilon = 0.0025$ and $\omega = 1$.
In our numerical tests we use equidistant grids with mesh size $h$. 
If piecewise linear functions are used,
the stabilization parameter for the state and adjoint equation is chosen to be
\begin{equation}  \label{eq:tau-ex1}
    \tau_e
    = \left\{ \begin{array}{ll}
         h^2/(4\epsilon),  & \Pe_e \le 1, \\
         h/2               & \Pe_e > 1.
       \end{array} \right.
\end{equation}
For piecewise quadratic finite elements the stabilization parameter for the state 
and adjoint equation is given by \eref{eq:tau-ex1} with $h$ replaced by $h/2$.

Errors and estimated convergence order for the
discretize-then-optimize approach as well as the
optimize-then-discretize approach using linear
($k = \ell = m = 1$) and quadratic ($k = \ell = m = 2$)
finite elements are given in Tables~\tref{tab:111-ex1}
and \tref{tab:222-ex1}.
In all examples we estimate the convergence order by taking the
logarithm with base two of the quotient of the error at grid size
$h$ and the error at grid size $h/2$.
In all examples the linear systems arising from the discretization
of the discretized optimal control problem
or from the discretization of the infinite dimensional optimality
conditions are solved using a sparse LU-decomposition.

For this example $\Pe_e = 1$ for $h = 5 \cdot 10^{-3}$, i.e.,
half of the data in Tables~\tref{tab:111-ex1}
and \tref{tab:222-ex1} correspond to the case $\Pe_e \le 1$.

If linear finite elements are used, i.e., if $k=m=\ell =1$,
Theorems \tref{th:disc-opt-err} and \tref{th:opt-disc-err}
predict that the error for states, controls and adjoints
behaves like $O(h)$ for $\Pe_e \le 1$. This is confirmed
by the results in Table~\tref{tab:111-ex1}. 
Table~\tref{tab:111-ex1} reveals few differences between the
discretize-then-optimize and the optimize-then-discretize
approach. If linear finite elements are used,
both produce states and adjoints that are of the same quality.
The controls computed using the discretize-then-optimize approach
are slightly better than the controls obtained from the
optimize-then-discretize approach. However, we have seen
examples where the opposite is true.

The situation changes if quadratic finite elements are used, i.e., 
if $k=m=\ell =2$. In this case Table~\tref{tab:222-ex1} shows
that convergence order for the adjoints computed using 
discretize-then-optimize approach is one, whereas the
convergence order for the adjoints obtained from the
optimize-then-discretize approach is two. This agrees
with the theoretical results in Theorems \tref{th:disc-opt-err},
\tref{th:disc-opt-adj} and in Theorem \tref{th:opt-disc-err},
respectively. 
However, in both cases the observed convergence order for the $L^2$ error
in the adjoints is one higher than the convergence order 
for the SUPG-error. The $L^2$ error for the controls
is much smaller than the SUPG-error in the states.
In fact, the term $h^k ( \epsilon^{1/2} +  h^{1/2}) | y |_{k+1}$
dominates $\| u_h-u_{\rm ex} \|_0$ and
Theorem~\tref{th:disc-opt-state} predicts that in the
discretize-then-optimize approach the states converge with
order two, instead of the pessimistic state error bound
provided by Theorem~\tref{th:disc-opt-err}.
For the optimize-then-discretize approach Theorem~\tref{th:opt-disc-err}
predicts that order of convergence in the SUPG error for both the states and
the adjoints is two. The  observed convergence order for the $L^2$ error
in the adjoints is one higher than the convergence order 
for the SUPG-error. Since the controls are multiples of the adjoints,
the observed convergence order for the $L^2$ error
in the controls is three, one higher than the convergence order prediced by
Theorem~\tref{th:opt-disc-err}.

\begin{table}[htb]
\caption{Errors and estimated convergence order. Example 1, $k = \ell = m = 1$.
         \label{tab:111-ex1}}
\begin{center}

The discretize-then-optimize approach

\begin{tabular}{c|c|c|c|c|c|c|c|c|c|c}
           &\multicolumn{4}{|c|}{$\|y_h-y_{\rm ex}\|$}
           &\multicolumn{2}{|c|}{$\|u_h-u_{\rm ex}\|$}
           &\multicolumn{4}{|c}{$\|\lambda_h-\lambda_{\rm ex}\|$} \\
     $h$   &$\|\cdot\|_0$&order&$\|\cdot\|_{SD}$&order
           &$\|\cdot\|_0$&order 
           &$\|\cdot\|_0$&order&$\|\cdot\|_{SD}$&order \\ \hline
 1.00e-1 & 1.86e-1  &     & 6.21e-1  &     & 7.76e-2  &     & 1.74e-1  &     & 6.20e-1  &     \\ 
 5.00e-2 & 1.24e-1 &  0.59 & 7.47e-1 & -0.27 & 4.45e-2 &  0.80 & 1.19e-1 &  0.54 & 7.47e-1 & -0.27 \\ 
 2.50e-2 & 7.69e-2 &  0.69 & 1.11e+0 & -0.57 & 2.18e-2 &  1.03 & 7.51e-2 &  0.67 & 1.11e+0 & -0.57 \\ 
 1.25e-2 & 4.65e-2 &  0.73 & 9.86e-1 &  0.17 & 1.37e-2 &  0.67 & 4.58e-2 &  0.71 & 9.87e-1 &  0.17 \\ 
 6.25e-3 & 2.64e-2 &  0.82 & 6.42e-1 &  0.62 & 1.03e-2 &  0.41 & 2.61e-2 &  0.81 & 6.43e-1 &  0.62 \\ 
 3.13e-3 & 9.58e-3 &  1.46 & 3.01e-1 &  1.09 & 5.05e-3 &  1.03 & 9.50e-3 &  1.46 & 3.01e-1 &  1.09 \\ 
 1.56e-3 & 2.57e-3 &  1.90 & 1.35e-1 &  1.16 & 1.55e-3 &  1.70 & 2.55e-3 &  1.90 & 1.35e-1 &  1.16 \\ 
 7.81e-4 & 6.54e-4 &  1.97 & 6.48e-2 &  1.06 & 4.15e-4 &  1.90 & 6.49e-4 &  1.97 & 6.48e-2 &  1.06
\end{tabular}
\end{center}

\begin{center}

The optimize-then-discretize approach

\begin{tabular}{c|c|c|c|c|c|c|c|c|c|c}
           &\multicolumn{4}{|c|}{$\|y_h-y_{\rm ex}\|$}
           &\multicolumn{2}{|c|}{$\|u_h-u_{\rm ex}\|$}
           &\multicolumn{4}{|c}{$\|\lambda_h-\lambda_{\rm ex}\|$} \\
     $h$   &$\|\cdot\|_0$&order&$\|\cdot\|_{SD}$&order
           &$\|\cdot\|_0$&order 
           &$\|\cdot\|_0$&order&$\|\cdot\|_{SD}$&order \\ \hline
 1.00e-1 & 1.85e-1 &       & 6.35e-1 &       & 1.83e-1 &       & 1.83e-1 &       & 6.76e-1 &     \\ 
 5.00e-2 & 1.23e-1 &  0.58 & 7.50e-1 & -0.24 & 1.23e-1 &  0.58 & 1.23e-1 &  0.58 & 7.58e-1 & -0.16 \\ 
 2.50e-2 & 7.66e-2 &  0.69 & 1.11e+0 & -0.57 & 7.63e-2 &  0.69 & 7.63e-2 &  0.69 & 1.11e+0 & -0.55 \\ 
 1.25e-2 & 4.63e-2 &  0.73 & 9.86e-1 &  0.17 & 4.61e-2 &  0.73 & 4.61e-2 &  0.73 & 9.86e-1 &  0.17 \\ 
 6.25e-3 & 2.63e-2 &  0.82 & 6.42e-1 &  0.62 & 2.62e-2 &  0.82 & 2.62e-2 &  0.82 & 6.41e-1 &  0.62 \\ 
 3.13e-3 & 9.56e-3 &  1.46 & 3.01e-1 &  1.09 & 9.52e-3 &  1.46 & 9.52e-3 &  1.46 & 3.01e-1 &  1.09 \\ 
 1.56e-3 & 2.56e-3 &  1.90 & 1.35e-1 &  1.16 & 2.55e-3 &  1.90 & 2.55e-3 &  1.90 & 1.35e-1 &  1.16 \\ 
 7.81e-4 & 6.52e-4 &  1.97 & 6.48e-2 &  1.06 & 6.50e-4 &  1.97 & 6.50e-4 &  1.97 & 6.47e-2 &  1.06 
\end{tabular}
\end{center}
\end{table}

\begin{table}[hbt]
\caption{Errors and estimated convergence order. Example 1, $k = \ell = m = 2$.
         \label{tab:222-ex1}}
\begin{center}

The discretize-then-optimize approach
\begin{tabular}{c|c|c|c|c|c|c|c|c|c|c}
           &\multicolumn{4}{|c|}{$\|y_h-y_{\rm ex}\|$}
           &\multicolumn{2}{|c|}{$\|u_h-u_{\rm ex}\|$}
           &\multicolumn{4}{|c}{$\|\lambda_h-\lambda_{\rm ex}\|$} \\
     $h$   &$\|\cdot\|_0$&order&$\|\cdot\|_{SD}$&order
           &$\|\cdot\|_0$&order 
           &$\|\cdot\|_0$&order&$\|\cdot\|_{SD}$&order \\ \hline
 1.00e-1 & 1.21e-1 &       & 6.11e-1 &       & 3.89e-2 &       & 1.21e-1 &     & 5.91e-1  &     \\ 
 5.00e-2 & 7.61e-2 &  0.67 & 3.62e-1 &  0.76 & 2.00e-2 &  0.96 & 7.99e-2 &  0.60 & 3.55e-1 &  0.74 \\ 
 2.50e-2 & 3.89e-2 &  0.97 & 5.15e-1 & -0.51 & 4.23e-3 &  2.24 & 4.70e-2 &  0.77 & 5.93e-1 & -0.74 \\ 
 1.25e-2 & 1.73e-2 &  1.17 & 4.44e-1 &  0.21 & 6.53e-3 & -0.63 & 2.92e-2 &  0.69 & 6.32e-1 & -0.09 \\ 
 6.25e-3 & 3.71e-3 &  2.22 & 1.65e-1 &  1.43 & 3.83e-3 &  0.77 & 1.22e-2 &  1.26 & 3.85e-1 &  0.71 \\ 
 3.13e-3 & 4.23e-4 &  3.13 & 4.21e-2 &  1.97 & 1.29e-3 &  1.57 & 3.54e-3 &  1.78 & 1.92e-1 &  1.00 \\ 
 1.56e-3 & 4.56e-5 &  3.21 & 1.04e-2 &  2.02 & 3.62e-4 &  1.83 & 9.27e-4 &  1.93 & 9.59e-2 &  1.00 \\ 
 7.81e-4 & 5.33e-6 &  3.10 & 2.58e-3 &  2.01 & 9.45e-5 &  1.94 & 2.35e-4 &  1.98 & 4.79e-2 &  1.00
\end{tabular}
\end{center}

\begin{center}

The optimize-then-discretize approach

\begin{tabular}{c|c|c|c|c|c|c|c|c|c|c}
           &\multicolumn{4}{|c|}{$\|y_h-y_{\rm ex}\|$}
           &\multicolumn{2}{|c|}{$\|u_h-u_{\rm ex}\|$}
           &\multicolumn{4}{|c}{$\|\lambda_h-\lambda_{\rm ex}\|$} \\
     $h$   &$\|\cdot\|_0$&order&$\|\cdot\|_{SD}$&order
           &$\|\cdot\|_0$&order 
           &$\|\cdot\|_0$&order&$\|\cdot\|_{SD}$&order \\ \hline
 1.00e-1 & 1.21e-1 &       & 6.19e-1 &       & 1.20e-1 &       & 1.20e-1 &     & 6.40e-1  &     \\ 
 5.00e-2 & 7.59e-2 &  0.67 & 3.64e-1 &  0.77 & 7.56e-2 &  0.67 & 7.56e-2 &  0.67 & 3.69e-1 &  0.79 \\ 
 2.50e-2 & 3.88e-2 &  0.97 & 5.14e-1 & -0.50 & 3.87e-2 &  0.97 & 3.87e-2 &  0.97 & 5.13e-1 & -0.48 \\ 
 1.25e-2 & 1.72e-2 &  1.17 & 4.44e-1 &  0.21 & 1.72e-2 &  1.17 & 1.72e-2 &  1.17 & 4.44e-1 &  0.21 \\ 
 6.25e-3 & 3.70e-3 &  2.22 & 1.65e-1 &  1.43 & 3.70e-3 &  2.22 & 3.70e-3 &  2.22 & 1.65e-1 &  1.43 \\ 
 3.13e-3 & 4.23e-4 &  3.13 & 4.21e-2 &  1.97 & 4.23e-4 &  3.13 & 4.23e-4 &  3.13 & 4.21e-2 &  1.97 \\ 
 1.56e-3 & 4.56e-5 &  3.21 & 1.04e-2 &  2.02 & 4.56e-5 &  3.21 & 4.56e-5 &  3.21 & 1.04e-2 &  2.02 \\ 
 7.81e-4 & 5.33e-6 &  3.10 & 2.58e-3 &  2.01 & 5.33e-6 &  3.10 & 5.33e-6 &  3.10 & 2.58e-3 &  2.01 
\end{tabular}
\end{center}
\end{table}

We have obtained qualitatively similar results when the choice 
\eref{eq:tau-ex1} of the stabilization parameter is replaced by
$\tau_e =  (|b| h/2) (\coth(\Pe_e) + 1/\Pe_e)$, which applied
to certain classes of state equations with fixed control gives
approximations that are nodally exact
\cite{ANBrooks_TJRHughes_1982a}, \cite[p.~234]{HGRoos_MStynes_LTobiska_1996}.

\newpage

%%%%%%%%%%%%%%%%%%%%%%%%%%%%%%%%%%%%%%%%%%%%%%%%%%%%%%%%%%%%%%%%%%%%%%%%
\subsection{Example 2}   \label{sec:ex2}

The second example is a two dimensional control problem 
on $\Omega = (-1,1)\times (0,1)$.
We use the data
\[
 \Gamma_n = (0,1)\times \{ 0 \}, \quad
 \Gamma_d = \partial \Omega \setminus \Gamma_n, \quad
 \bc(x) = \left(\begin{array}{c}
               2x_2(1-x_1^2) \\ -2x_1(1-x_2^2)
          \end{array} \right), \quad
  r = 0, 
\]
$\epsilon = 10^{-5}$ and $\omega = 10^{-2}$.
The functions $f$, $d$, $g$ and $\widehat{y}$ are chosen so that the solution
of the optimal control problem \eref{eq:intro-obj}, \eref{eq:intro-state}
is given by
\[
    y_{\rm ex}(x)       =  1 + \tanh(1-(2(x_1^2+x_2^2)^{1/2}+1)), \quad
    \lambda_{\rm ex}(x) = (x_1^2-1)x_2^2(x_2-1)
\]  
and $u_{\rm ex} = \omega^{-1}  \lambda_{\rm ex}$.
This example is motivated by the first model problem in
\cite[pp.~9,10]{KWMorton_1996}.
Our triangulation is computed by first subdividing $\Omega$ into
squares of size $h \times h$ and then dividing each square
into two triangles. 
If piecewise linear functions are used,
the stabilization parameter for the state and adjoint equation is chosen to be
\begin{equation}  \label{eq:tau-ex2}
    \tau_e
    = \left\{ \begin{array}{ll}
         h^2/(4\epsilon),     & \Pe_e \le 1, \\
         h/(2\| \bc \|_{0,\infty,T_e})             & \Pe_e > 1.
       \end{array} \right.
\end{equation}
For piecewise quadratic finite elements the stabilization parameter for the state 
and adjoint equation is given by \eref{eq:tau-ex1} with $h$ replaced by $h/2$.

\begin{table}[hbt]
\caption{Errors and estimated convergence order. Example 2, $k = \ell = m = 1$.
         \label{tab:111-ex2}}
\begin{center}

The discretize-then-optimize approach

\begin{tabular}{c|c|c|c|c|c|c|c|c|c|c}
           &\multicolumn{4}{|c|}{$\|y_h-y_{\rm ex}\|$}
           &\multicolumn{2}{|c|}{$\|u_h-u_{\rm ex}\|$}
           &\multicolumn{4}{|c}{$\|\lambda_h-\lambda_{\rm ex}\|$} \\
     $h$   &$\|\cdot\|_0$&order&$\|\cdot\|_{SD}$&order
           &$\|\cdot\|_0$&order 
           &$\|\cdot\|_0$&order&$\|\cdot\|_{SD}$&order \\ \hline
 2.00e-1 & 1.69e-1 &       & 1.12e+0 &       & 8.27e-1 &       & 1.67e-2 &       & 5.32e-2 &       \\ 
 1.00e-1 & 8.52e-2 &  0.99 & 6.11e-1 &  0.87 & 4.27e-1 &  0.95 & 9.18e-3 &  0.86 & 2.80e-2 &  0.93 \\ 
 5.00e-2 & 2.35e-2 &  1.86 & 2.58e-1 &  1.24 & 1.77e-1 &  1.27 & 4.60e-3 &  1.00 & 1.20e-2 &  1.22 \\ 
 2.50e-2 & 4.60e-3 &  2.35 & 9.50e-2 &  1.44 & 4.69e-2 &  1.92 & 2.21e-3 &  1.06 & 3.47e-3 &  1.79 \\ 
 1.25e-2 & 8.50e-4 &  2.44 & 3.35e-2 &  1.50 & 7.75e-3 &  2.60 & 1.09e-3 &  1.02 & 8.66e-4 &  2.00 \\ 
 6.25e-3 & 1.96e-4 &  2.12 & 1.18e-2 &  1.51 & 1.17e-3 &  2.73 & 5.48e-4 &  0.99 & 2.82e-4 &  1.62 
\end{tabular}
\end{center}

\begin{center}

The optimize-then-discretize approach

\begin{tabular}{c|c|c|c|c|c|c|c|c|c|c}
           &\multicolumn{4}{|c|}{$\|y_h-y_{\rm ex}\|$}
           &\multicolumn{2}{|c|}{$\|u_h-u_{\rm ex}\|$}
           &\multicolumn{4}{|c}{$\|\lambda_h-\lambda_{\rm ex}\|$} \\
     $h$   &$\|\cdot\|_0$&order&$\|\cdot\|_{SD}$&order
           &$\|\cdot\|_0$&order 
           &$\|\cdot\|_0$&order&$\|\cdot\|_{SD}$&order \\ \hline
 2.00e-1 & 1.76e-1 &       & 1.12e+0 &       & 7.51e-1 &       & 7.51e-3 &       & 4.78e-2 &       \\ 
 1.00e-1 & 8.64e-2 &  1.03 & 6.12e-1 &  0.87 & 4.14e-1 &  0.86 & 4.14e-3 &  0.86 & 2.59e-2 &  0.88 \\ 
 5.00e-2 & 2.38e-2 &  1.86 & 2.59e-1 &  1.24 & 1.74e-1 &  1.25 & 1.74e-3 &  1.25 & 1.13e-2 &  1.20 \\ 
 2.50e-2 & 4.64e-3 &  2.36 & 9.50e-2 &  1.45 & 4.66e-2 &  1.90 & 4.66e-4 &  1.90 & 3.13e-3 &  1.85 \\ 
 1.25e-2 & 8.55e-4 &  2.44 & 3.35e-2 &  1.50 & 7.72e-3 &  2.59 & 7.72e-5 &  2.59 & 6.75e-4 &  2.21 \\ 
 6.25e-3 & 1.97e-4 &  2.12 & 1.18e-2 &  1.51 & 1.17e-3 &  2.72 & 1.17e-5 &  2.72 & 2.04e-4 &  1.73 
\end{tabular}
\end{center}
\end{table}

\begin{table}[hbt]
\caption{Errors and estimated convergence order. Example 2, $k = \ell = m = 2$.
         \label{tab:222-ex2}}
\begin{center}

The discretize-then-optimize approach

\begin{tabular}{c|c|c|c|c|c|c|c|c|c|c}
           &\multicolumn{4}{|c|}{$\|y_h-y_{\rm ex}\|$}
           &\multicolumn{2}{|c|}{$\|u_h-u_{\rm ex}\|$}
           &\multicolumn{4}{|c}{$\|\lambda_h-\lambda_{\rm ex}\|$} \\
     $h$   &$\|\cdot\|_0$&order&$\|\cdot\|_{SD}$&order
           &$\|\cdot\|_0$&order 
           &$\|\cdot\|_0$&order&$\|\cdot\|_{SD}$&order \\ \hline
 2.00e-1 & 5.60e-2 &       & 6.64e-1 &       & 3.22e-1 &       & 8.34e-3 &       & 2.91e-2 &       \\
 1.00e-1 & 1.43e-2 &  1.97 & 2.38e-1 &  1.48 & 9.68e-2 &  1.73 & 4.24e-3 &  0.98 & 1.07e-2 &  1.45 \\
 5.00e-2 & 1.92e-3 &  2.89 & 5.41e-2 &  2.14 & 1.62e-2 &  2.58 & 2.14e-3 &  0.99 & 2.77e-3 &  1.94 \\
 2.50e-2 & 3.52e-4 &  2.45 & 1.28e-2 &  2.08 & 2.78e-3 &  2.54 & 1.08e-3 &  0.98 & 1.33e-3 &  1.06 \\
 1.25e-2 & 1.27e-4 &  1.47 & 8.01e-3 &  0.67 & 2.75e-3 &  0.02 & 5.46e-4 &  0.99 & 1.99e-3 & -0.58 
\end{tabular}
\end{center}

\begin{center}

The optimize-then-discretize approach

\begin{tabular}{c|c|c|c|c|c|c|c|c|c|c}
           &\multicolumn{4}{|c|}{$\|y_h-y_{\rm ex}\|$}
           &\multicolumn{2}{|c|}{$\|u_h-u_{\rm ex}\|$}
           &\multicolumn{4}{|c}{$\|\lambda_h-\lambda_{\rm ex}\|$} \\
     $h$   &$\|\cdot\|_0$&order&$\|\cdot\|_{SD}$&order
           &$\|\cdot\|_0$&order 
           &$\|\cdot\|_0$&order&$\|\cdot\|_{SD}$&order \\ \hline
 2.00e-1 & 5.74e-2 &       & 6.66e-1 &       & 3.08e-1 &       & 3.08e-3 &       & 2.54e-2 &       \\
 1.00e-1 & 1.45e-2 &  1.99 & 2.39e-1 &  1.48 & 9.43e-2 &  1.71 & 9.43e-4 &  1.71 & 8.97e-3 &  1.50 \\
 5.00e-2 & 1.92e-3 &  2.91 & 5.39e-2 &  2.15 & 1.53e-2 &  2.62 & 1.53e-4 &  2.62 & 1.59e-3 &  2.49 \\
 2.50e-2 & 3.26e-4 &  2.56 & 1.12e-2 &  2.26 & 1.26e-3 &  3.61 & 1.26e-5 &  3.61 & 1.49e-4 &  3.41 \\
 1.25e-2 & 5.66e-5 &  2.53 & 2.23e-3 &  2.33 & 1.23e-4 &  3.35 & 1.23e-6 &  3.35 & 1.97e-5 &  2.92
\end{tabular}
\end{center}
\end{table}

Errors and estimated convergence order for the
discretize-then-optimize approach as well as the
optimize-then-discretize approach using linear
($k = \ell = m = 1$) and quadratic ($k = \ell = m = 2$)
finite elements are given in Tables~\tref{tab:111-ex2}
and \tref{tab:222-ex2}.
All the data in Tables~\tref{tab:111-ex2}
and \tref{tab:222-ex2} correspond to the case $\Pe_e > 1$.
The sizes of the smallest and largest systems 
\eref{eq:model-disc-opt-cond} and \eref{eq:model-opt-cond-disc}
arising in our calculations are $198\times 198$ and 
$155043 \times 155043$, respectively.  
To avoid contamination of the convergence errors by the truncation
of an iterative scheme, these systems were solved using a sparse LU-decomposition.

In this example our exact adjoints and controls are designed to be functions with
small gradients. If linear finite element approximations are used, the observed
convergence order for the SUPG error in the computed ajoints and the 
computed states is greater than one for both approaches.
See Table~\tref{tab:111-ex2}.
The observed convergence order for the $L^2$-error in the computed adoints
is only one for the discretze-then-optimize approach, while for
the optimize-then-discretize approach the  observed convergence order for 
the $L^2$-error in the computed adoints is one higher than the
observed convergence order for the SUPG-error. In this example
the optimize-then-discretize approach produced better approximations.

If quadratic finite elements are used, i.e., if $k=m=\ell =2$,
then the optimize-then-discretize approach leads to superior results.
See Table~\tref{tab:222-ex2}. For this approach, the observed
convergence order for the SUPG error in the computed ajoints and the 
computed states is greater than the guaranteed convergence order of two.
The  observed convergence order for the $L^2$-error in the computed adoints 
is one higher than the observed convergence order for the SUPG-error,
which leads to an observed convergence order greater than three for
the $L^2$-error in the controls.
This is different for discretze-then-optimize approach. Initially
the observed convergence orders for the  SUPG-errors in states and adjoints
as well as the $L^2$-error in control are comperable to those achieved
by the optimize-then-discretize approach, but deteriorates subsequently.

We note that in this example the gradients of $y_h$ are almost perpendicular
to $\bc$. Because of this feature, the standard Galerkin finite element method 
produced good solutions to the optimal control problem.

%%%%%%%%%%%%%%%%%%%%%%%%%%%%%%%%%%%%%%%%%%%%%%%%%%%%%%%%%%%%%%%%%%%%%%%%
\subsection{Example 3}   \label{sec:ex3}

In our third example we use $\Omega = (0,1)^2$ and
\[
 \Gamma_d = \partial \Omega, \quad 
 \bc(x) = \left(\cos(\theta) \; \sin(\theta) \right)^T, \;
  \theta = 45^\circ, \quad
  r = 0, 
\]
$\epsilon = 10^{-2}$ and $\omega = 1$.
The functions $f$, $d$, $g$ and $\widehat{y}$ are chosen so that the solution
of the optimal control problem \eref{eq:intro-obj}, \eref{eq:intro-state}
is given by
\[
    y_{\rm ex}(x)       = \eta(x_1) \eta(x_2) , \quad
    \lambda_{\rm ex}(x) = \mu(x_1) \mu(x_2),
\]  
where
\[
    \eta(z) =  z - \frac{ \exp((z-1)/\epsilon ) - \exp(-1/\epsilon) }
                         { 1 - \exp(-1/\epsilon) }, \quad
    \mu(z)
     = \left(1-z - \frac{\exp(-z/\epsilon)-\exp(-1/\epsilon)}
                        {1 - \exp(-1/\epsilon)} \right)
\]  
and $u_{\rm ex} = \omega^{-1}  \lambda_{\rm ex}$.
Our triangulation is computed by first subdividing $\Omega$ into
squares of size $h \times h$ and then dividing each square
into two triangles. Our choice for the stabilization
parameter is the same as in Section \tref{sec:ex2}.
Errors and estimated convergence order for the
discretize-then-optimize approach as well as the
optimize-then-discretize approach using linear
($k = \ell = m = 1$) and quadratic ($k = \ell = m = 2$)
finite elements are given in Tables~\tref{tab:111-ex3}
and \tref{tab:222-ex3}.
All but the last row in Table~\tref{tab:111-ex3} and all data
in Table~\tref{tab:222-ex3} correspond to the case $\Pe_e > 1$.
The sizes of the smallest and largest systems 
\eref{eq:model-disc-opt-cond} and \eref{eq:model-opt-cond-disc}
arising in our calculations are $363\times 363$ and 
$77763 \times 77763$, respectively.  
These systems were solved using a sparse LU-decomposition.

\begin{figure}[thb]
\begin{minipage}[t]{0.48\textwidth}
\begin{center}
\includegraphics[width=\textwidth]{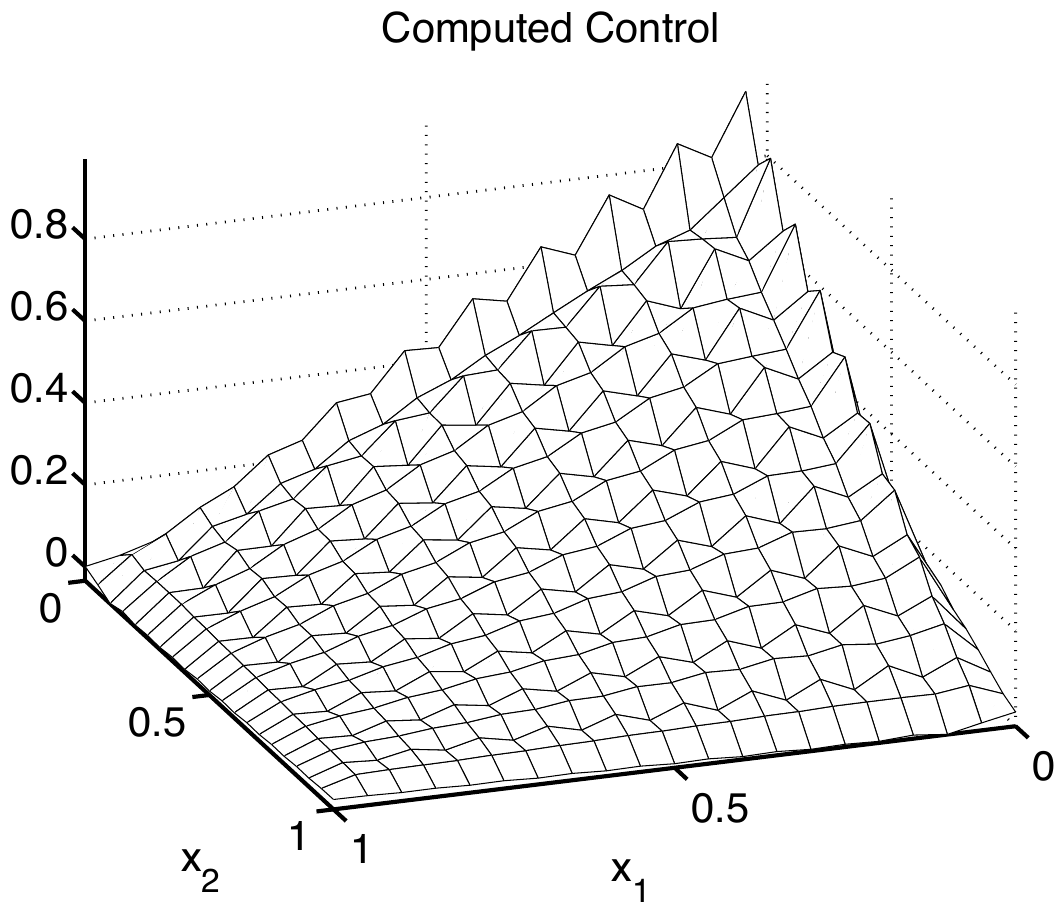}
%\epsfxsize=\textwidth
%\centerline{\epsfbox{./figures/control_ex3_nx10_basis2_disc1.eps}}
discretize-then-optimize
\end{center}
\end{minipage}\hfil
\begin{minipage}[t]{0.48\textwidth}
\begin{center}
\includegraphics[width=\textwidth]{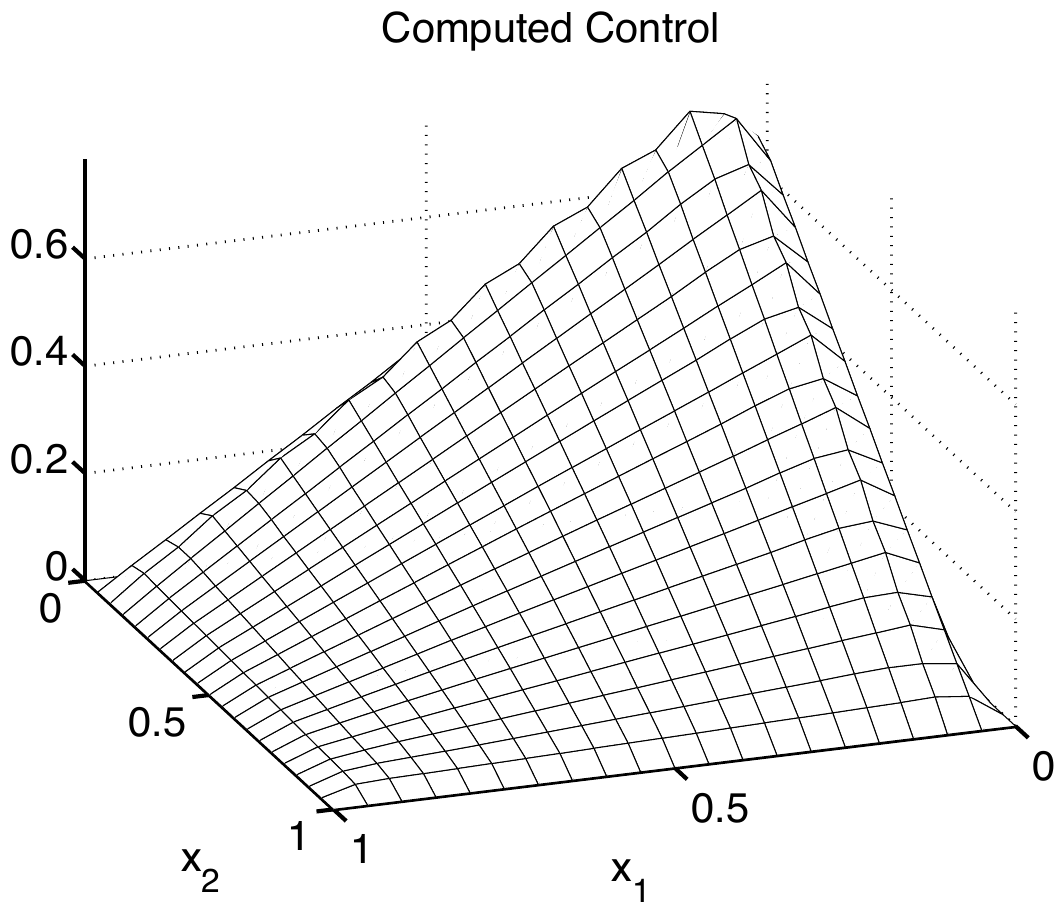}
%\epsfxsize=\textwidth
%\centerline{\epsfbox{./figures/control_ex3_nx10_basis2_disc0.eps}}
optimize-then-discretize
\end{center}
\end{minipage}

\caption{Computed controls, Example 3, $k = \ell = m = 2$, $h=0.1$.
 \label{fig:control_ex3_nx10_basis2}}

\end{figure}

\begin{figure}[hbt]
\begin{minipage}[t]{0.48\textwidth}
\begin{center}
\includegraphics[width=\textwidth]{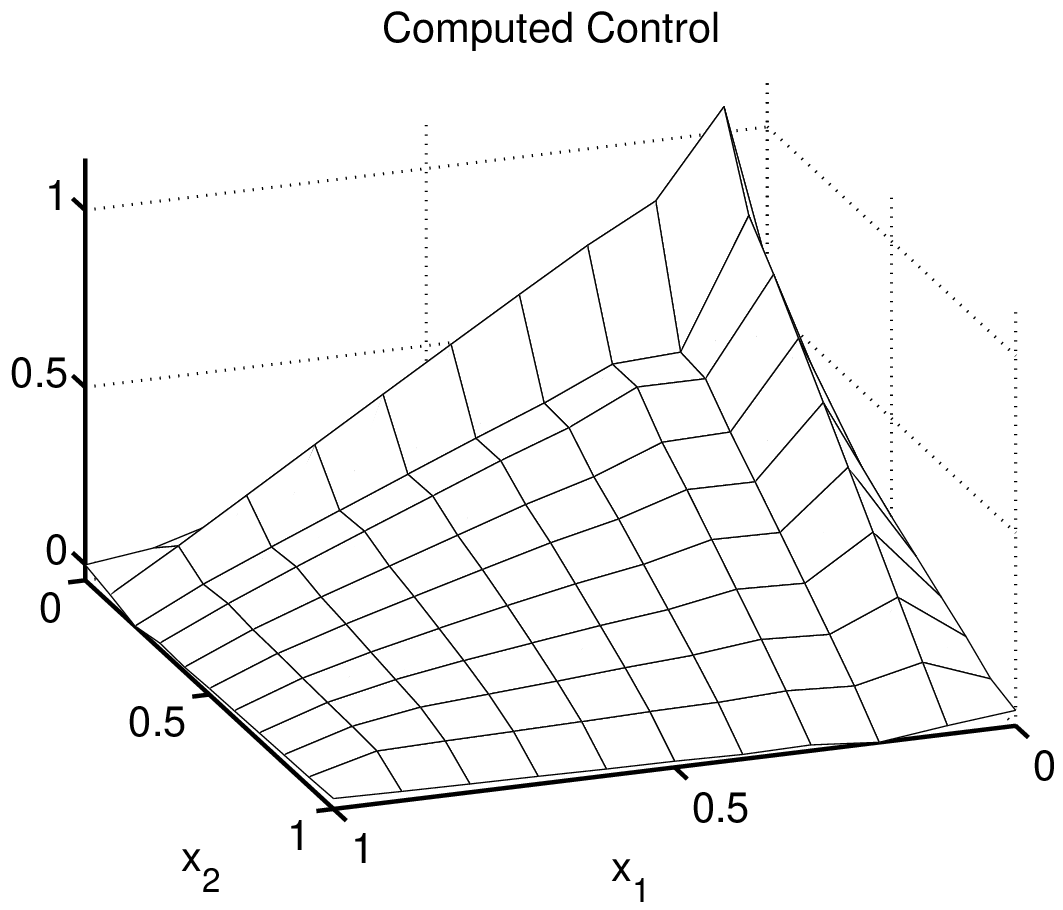}
%\epsfxsize=\textwidth
%\centerline{\epsfbox{./figures/control_ex3_nx10_basis1_disc1.eps}}
discretize-then-optimize
\end{center}
\end{minipage}\hfil
\begin{minipage}[t]{0.48\textwidth}
\begin{center}
\includegraphics[width=\textwidth]{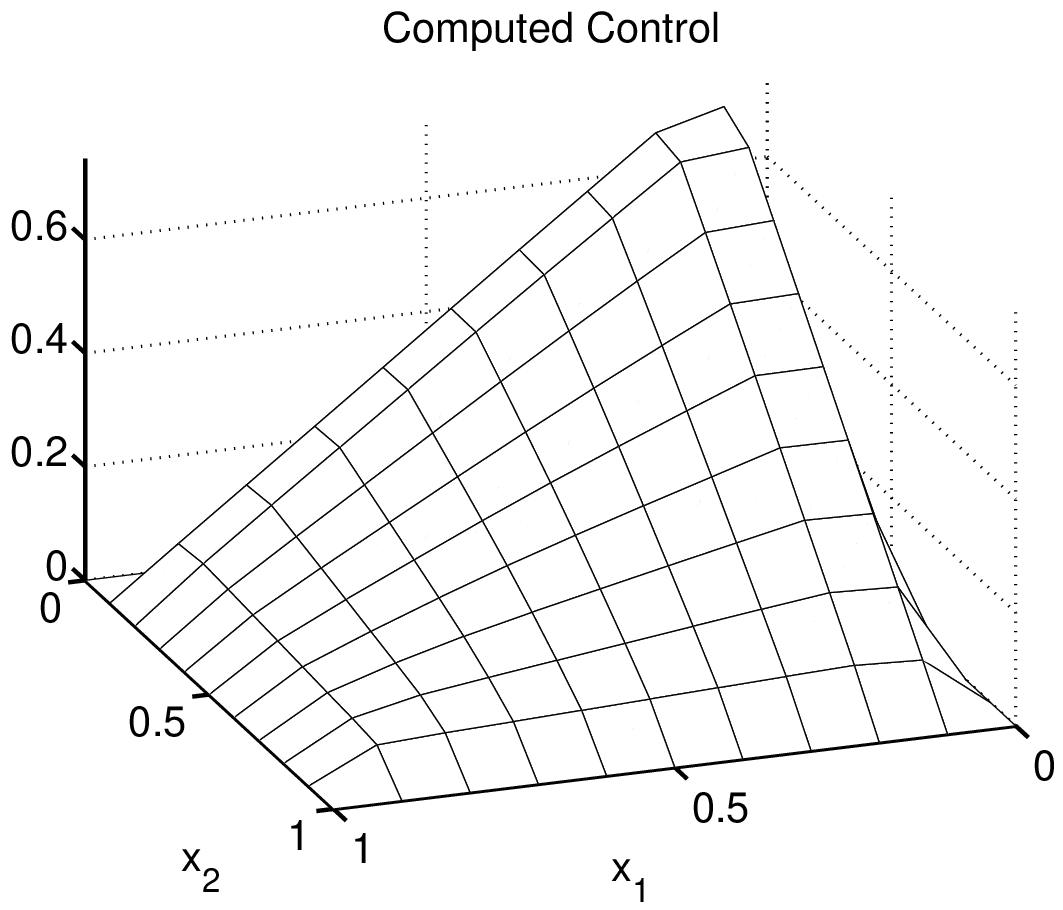}
%\epsfxsize=\textwidth
%\centerline{\epsfbox{./figures/control_ex3_nx10_basis1_disc0.eps}}
optimize-then-discretize
\end{center}
\end{minipage}

\caption{Computed controls, Example 3, $k = \ell = m = 1$, $h=0.1$.
 \label{fig:control_ex3_nx10_basis1}}

\end{figure}

The observations drawn from Tables~\tref{tab:111-ex3} and
\tref{tab:222-ex3} for this example are very similar to those of Example
1. However, when quadratic finite elements are used, we observed small
node-to-node oscillations in the adjoints and controls computed by the
discretize-then optimize approach. These are not present in the adjoints
and controls computed by the optimize-then-discretize approach.  See
Figure \tref{fig:control_ex3_nx10_basis2}.  Such node-to-node
oscillations in the adjoints and controls did not develop in either
approach when linear finite elements are used as seen in Figure
\tref{fig:control_ex3_nx10_basis1}.  For better visibility we show the
results on a coarse grid, but the plots of our results on finer grid are
qualitatively comparable to those in Figures
\tref{fig:control_ex3_nx10_basis2} and
\tref{fig:control_ex3_nx10_basis1}.

We remark that for this example the standard Galerkin method 
produced poor results. For smaller diffusion $\epsilon$ even
SUPG using either approach did not produce satisfactory
approximations to the optimal control, states and adjoints for
coarser grids.

\begin{table}[htb]
\caption{Errors and estimated convergence order. Example 3, $k = \ell = m = 1$.
         \label{tab:111-ex3}}
\begin{center}

The discretize-then-optimize approach

\begin{tabular}{c|c|c|c|c|c|c|c|c|c|c}
           &\multicolumn{4}{|c|}{$\|y_h-y_{\rm ex}\|$}
           &\multicolumn{2}{|c|}{$\|u_h-u_{\rm ex}\|$}
           &\multicolumn{4}{|c}{$\|\lambda_h-\lambda_{\rm ex}\|$} \\
     $h$   &$\|\cdot\|_0$&order&$\|\cdot\|_{SD}$&order
           &$\|\cdot\|_0$&order 
           &$\|\cdot\|_0$&order&$\|\cdot\|_{SD}$&order \\ \hline
 1.00e-1 & 1.20e-1 &       & 1.25e+0 &       & 4.91e-2 &       & 1.09e-1 &       & 1.26e+0 &       \\ 
 5.00e-2 & 7.43e-2 &  0.69 & 1.11e+0 &  0.17 & 2.48e-2 &  0.98 & 6.31e-2 &  0.79 & 1.09e+0 &  0.20 \\ 
 2.50e-2 & 4.17e-2 &  0.83 & 7.58e-1 &  0.55 & 1.36e-2 &  0.87 & 3.28e-2 &  0.94 & 7.22e-1 &  0.60 \\ 
 1.25e-2 & 1.46e-2 &  1.51 & 3.83e-1 &  0.99 & 5.88e-3 &  1.21 & 1.13e-2 &  1.54 & 3.69e-1 &  0.97 \\ 
 6.25e-3 & 3.86e-3 &  1.92 & 1.83e-1 &  1.07 & 1.68e-3 &  1.81 & 2.99e-3 &  1.92 & 1.81e-1 &  1.03 
\end{tabular}
\end{center}

\begin{center}

The optimize-then-discretize approach

\begin{tabular}{c|c|c|c|c|c|c|c|c|c|c}
           &\multicolumn{4}{|c|}{$\|y_h-y_{\rm ex}\|$}
           &\multicolumn{2}{|c|}{$\|u_h-u_{\rm ex}\|$}
           &\multicolumn{4}{|c}{$\|\lambda_h-\lambda_{\rm ex}\|$} \\
     $h$   &$\|\cdot\|_0$&order&$\|\cdot\|_{SD}$&order
           &$\|\cdot\|_0$&order 
           &$\|\cdot\|_0$&order&$\|\cdot\|_{SD}$&order \\ \hline
 1.00e-1 & 1.22e-1 &       & 1.25e+0 &       & 1.18e-1 &       & 1.18e-1 &       & 1.25e+0 &       \\ 
 5.00e-2 & 7.54e-2 &  0.70 & 1.11e+0 &  0.17 & 7.36e-2 &  0.69 & 7.36e-2 &  0.69 & 1.11e+0 &  0.17 \\ 
 2.50e-2 & 4.22e-2 &  0.84 & 7.59e-1 &  0.55 & 4.12e-2 &  0.83 & 4.12e-2 &  0.83 & 7.56e-1 &  0.55 \\ 
 1.25e-2 & 1.47e-2 &  1.52 & 3.83e-1 &  0.99 & 1.45e-2 &  1.51 & 1.45e-2 &  1.51 & 3.82e-1 &  0.98 \\ 
 6.25e-3 & 3.88e-3 &  1.92 & 1.83e-1 &  1.07 & 3.82e-3 &  1.92 & 3.82e-3 &  1.92 & 1.83e-1 &  1.06 
\end{tabular}
\end{center}
\end{table}

\begin{table}[hbt]
\caption{Errors and estimated convergence order. Example 3, $k = \ell = m = 2$.
         \label{tab:222-ex3}}
\begin{center}

The discretize-then-optimize approach

\begin{tabular}{c|c|c|c|c|c|c|c|c|c|c}
           &\multicolumn{4}{|c|}{$\|y_h-y_{\rm ex}\|$}
           &\multicolumn{2}{|c|}{$\|u_h-u_{\rm ex}\|$}
           &\multicolumn{4}{|c}{$\|\lambda_h-\lambda_{\rm ex}\|$} \\
     $h$   &$\|\cdot\|_0$&order&$\|\cdot\|_{SD}$&order
           &$\|\cdot\|_0$&order 
           &$\|\cdot\|_0$&order&$\|\cdot\|_{SD}$&order \\ \hline
 2.00e-1 & 1.13e-1 &       & 6.56e-1 &       & 4.19e-2 &       & 1.09e-1 &       & 7.64e-1 &       \\
 1.00e-1 & 6.07e-2 &  0.89 & 9.12e-1 & -0.48 & 1.60e-2 &  1.39 & 5.96e-2 &  0.87 & 9.50e-1 & -0.31 \\
 5.00e-2 & 2.73e-2 &  1.15 & 7.56e-1 &  0.27 & 1.29e-2 &  0.31 & 3.60e-2 &  0.73 & 1.04e+0 & -0.13 \\
 2.50e-2 & 5.42e-3 &  2.33 & 2.92e-1 &  1.37 & 6.53e-3 &  0.98 & 1.47e-2 &  1.29 & 7.05e-1 &  0.56 \\
 1.25e-2 & 6.35e-4 &  3.09 & 8.09e-2 &  1.85 & 2.16e-3 &  1.60 & 4.29e-3 &  1.78 & 3.96e-1 &  0.83
\end{tabular}
\end{center}

\begin{center}

The optimize-then-discretize approach

\begin{tabular}{c|c|c|c|c|c|c|c|c|c|c}
           &\multicolumn{4}{|c|}{$\|y_h-y_{\rm ex}\|$}
           &\multicolumn{2}{|c|}{$\|u_h-u_{\rm ex}\|$}
           &\multicolumn{4}{|c}{$\|\lambda_h-\lambda_{\rm ex}\|$} \\
     $h$   &$\|\cdot\|_0$&order&$\|\cdot\|_{SD}$&order
           &$\|\cdot\|_0$&order 
           &$\|\cdot\|_0$&order&$\|\cdot\|_{SD}$&order \\ \hline
 2.00e-1 & 1.14e-1 &       & 6.54e-1 &       & 1.13e-1 &       & 1.13e-1 &       & 6.78e-1 &       \\
 1.00e-1 & 6.13e-2 &  0.90 & 9.14e-1 & -0.48 & 6.05e-2 &  0.90 & 6.05e-2 &  0.90 & 9.07e-1 & -0.42 \\
 5.00e-2 & 2.76e-2 &  1.15 & 7.57e-1 &  0.27 & 2.72e-2 &  1.15 & 2.72e-2 &  1.15 & 7.54e-1 &  0.27 \\
 2.50e-2 & 5.43e-3 &  2.34 & 2.92e-1 &  1.38 & 5.42e-3 &  2.33 & 5.42e-3 &  2.33 & 2.92e-1 &  1.37 \\
 1.25e-2 & 6.35e-4 &  3.10 & 8.09e-2 &  1.85 & 6.35e-4 &  3.09 & 6.35e-4 &  3.09 & 8.09e-2 &  1.85 
\end{tabular}
\end{center}
\end{table}

\section{Conclusions}   \label{sec:concl}

We have studied the effect of the SUPG finite element method on the
discretization of optimal control problems governed by the linear
advection-diffusion equation.  Two approaches for the computation of
approximate controls and corresponding states and adjoints were compared: The
discretize-then-optimize approach and the optimize-then-discretize
approach.  Theoretical and numerical studies of the error between the exact
solution of the control problem and its approximation were provided.  Our
theoretical results show that the optimize-then-discretize approach leads to
asymptotically better approximate solutions than the
discretize-then-optimize approach.  The theoretical results also indicate
that the differences in solution quality is small when piecewise linear
polynomials are used for the discretization of states, adjoints and controls,
but that they can be significant if higher-order finite elements are used for
the states and the adjoints. There is always a significant difference in
quality of the adjoints computed by the discretize-then-optimize approach
and the optimize-then-discretize approach if finite element approximations
with polynomial degree greater than one are used, and the
optimize-then-discretize approach leads to better adjoint approximations.
However, our numerical results have also shown that this large difference in
the quality of adjoints does not necessarily imply a large difference in the
quality of the controls. 

Often the observed error in the controls computed by the
discretize-then-optimize approach and the optimize-then-discretize
approach is rather similar -- even if the adjoints computed using both
approaches are significantly different. This seems to be related to the fact
that we consider distributed controls and that errors in the controls are
measured in the $L^2$ norm whereas errors in the adjoints are measured in the
SUPG-norm. Since the distributed controls are multiples of the adjoints, our
numerical results indicate that the $L^2$-error in the adjoints is much
smaller than the error in the SUPG-norm.  Whether these good convergence
properties in the control also materialize if Neumann or Dirichlet boundary
controls are used or if other objective functionals acting on the state are
given is part of future studies.  Another subject of future study is the
influence of stabilization methods on the control of systems of advection
diffusion equations, like the Navier-Stokes equations, where additional
inconsistencies in the discretize-then-optimize approach can occur.

We conclude by reiterating  that care is required when using the
SUPG method for the solution of optimal control problems.
If the discretize-then-optimize approach with SUPG is used, the order with 
which the computed solutions of the optimal control converge to the exact 
solution may be much lower than what one would expect from the solution of 
a single advection diffusion equation using SUPG. The asymptotic
convergence behavior expected from the SUPG method applied to a single
advection diffusion equation can be maintained for the optimal control
problem if the optimize-then-discretize approach is used.

\section*{Acknowledgements}
The authors would like to thank Prof.\ M.~Behr,
Dept.\ of Mechanical Engineering, Rice University, for valuable
discussions on stabilization methods.

\bibstyle{siam}
\bibliography{/home/heinken/tex/bib/heinken,/home/heinken/tex/bib/references}
\bibliographystyle{siam}

%\clearpage \newpage \input{appendix1}

\end{document}